\documentclass{amsart}
\usepackage[mathscr]{eucal}
\usepackage{graphicx}
\usepackage{amscd}
\usepackage{amsmath}
\usepackage{amsthm}
\usepackage{amsxtra}
\usepackage{calc}
\usepackage{amsfonts}
\usepackage{amssymb}
\usepackage{latexsym}
\usepackage{pdfsync}
\usepackage{amsopn}
\usepackage{mathrsfs}
\usepackage{accents}
\usepackage{enumitem}
\usepackage[all]{xy}
\SelectTips{cm}{}
\calclayout

\usepackage[colorlinks=true,linkcolor=red,citecolor=green,bookmarks=true]{hyperref}

\newcommand{\lra}{\longrightarrow}

\def\xx{\text{\boldmath $x$}}
\def\yy{\text{\boldmath $y$}}

\newcommand{\md}{\mathsf{mod}}
\newcommand{\Md}{\mathsf{Mod}}

\newcommand{\rep}{\mathsf{rep}}
\newcommand{\st}{\stackrel}

\newcommand{\im}{{\rm{Im}}}

\newcommand{\add} {\mathsf{add}}

\newcommand{\m}{{\mathfrak{m}}}

\newcommand{\p}{{\mathfrak{p}}}
\newcommand{\cok}{{\rm{Coker}}}
\newcommand{\Ker}{{\rm{Ker}}}

\newcommand{\Tr}{\mathrm{Tr}}
\newcommand{\mcm}{\rm{mCM}}

\newcommand{\ra}{\mathsf{rad}}
\newcommand{\op}{\rm{op}}

\newcommand{\Hom}{{\mathsf{Hom}}}
\newcommand{\h}{{\mathsf{h}}}

\newcommand{\End}{{\mathsf{End}}}

\newcommand{\Ext}{\mathsf{{Ext}}}
\newcommand{\F}{\mathcal{{F}}}
\newcommand{\X}{\mathcal{{X}}}

\DeclareMathOperator{\ind}{\mathtt{ind}}

\theoremstyle{definition}
\newtheorem{theorem}{Theorem}[section]

\newtheorem{cor}[theorem]{Corollary}

\newtheorem{lemma}[theorem]{Lemma}
\newtheorem{prop}[theorem]{Proposition}

\newtheorem{definition}[theorem]{Definition}

\newtheorem{remark}[theorem]{Remark}

\theoremstyle{plain}

\theoremstyle{definition}

\numberwithin{equation}{section}

\newcounter{hours}
\newcounter{minutes}

\begin{document}

\title[orders of bounded and strongly unbounded lattice type]
{orders of bounded and strongly unbounded lattice type}

\author[Fotouhi, Martsinkovsky and Salarian] {Fahimeh Sadat Fotouhi, Alex Martsinkovsky, and Shokrollah Salarian}

\address{Department of Mathematics, University of Isfahan, P.O.Box: 81746-73441, Isfahan, Iran }
 \email{fsfotuhi@sci.ui.ac.ir}
 
\address{Northeastern University
Department of Mathematics
360 Huntington Avenue,
Boston, MA 02115, USA}
 \email{alexmart@neu.edu}

\address{Department of Mathematics, University of Isfahan, P.O.Box: 81746-73441, Isfahan,
 Iran and \\ School of Mathematics, Institute for Research in Fundamental Science (IPM), P.O.Box: 19395-5746, Tehran, Iran}
 \email{Salarian@ipm.ir}

\date{\today, \setcounter{hours}{\time/60} \setcounter{minutes}{\time-\value{hours}*60} \thehours\,h\ \theminutes\,min}



\thanks{}

\begin{abstract}
Brauer and Thrall conjectured that a finite-dimensional algebra over a field of  bounded representation type is actually of finite representation type and a finite-dimensional algebra (over an infinite field) of infinite representation type has strongly unbounded representation type. These conjectures, now theorems, are our motivation for studying (generalized) orders of bounded and strongly unbounded lattice type. To each lattice over an order we assign a numerical invariant, $\underline{\h}$-length, measuring Hom modulo projectives. We show that an order of bounded lattice type is actually of finite lattice type, and if there are infinitely many non-isomorphic  indecomposable lattices of the same $\underline{\h}$-length, then the order has strongly unbounded lattice type.

For a hypersurface $R=k[[x_0,...,x_d]]/(f)$, we show that $R$ is of bounded (respectively, strongly unbounded) lattice type if and only if the double branched cover $R^{\sharp}$ of $R$ is of bounded (respectively, strongly unbounded) lattice type. This is an analog of a result of Kn\"{o}rrer  and Buchweitz-Greuel-Schreyer for rings of finite  mCM type. Consequently, it is proved that $R$ has strongly unbounded lattice type whenever $k$ is infinite.
\end{abstract}

\maketitle

\tableofcontents

\section{Introduction}
In~\cite{jans1957indecomposable}, Jans states that R. Brauer and R. M. Thrall  conjectured that a finite-dimensional algebra 
%$A$ 
over a field $k$ of  bounded representation type (meaning that
there is a bound on the lengths of the indecomposable finitely generated modules) is actually of finite representation type. They also conjectured that a finite-dimensional algebra over an infinite field of infinite representation type has strongly unbounded representation type (meaning that there is an infinite sequence $n_1<n_2<\cdots$  of positive integers such that there are, for any $i$, infinitely many non-isomorphic indecomposable modules of $k$-dimension $n_i$).
%The Brauer-Thrall conjectures became one of the main problems in representation theory of algebras, and in particular,
%they have been a guideline for many investigation in the representation theory of artin algebras.
Both conjectures are now theorems. The first Brauer-Thrall conjecture was proved by Roiter \cite{roiter1968unbounded}. Later, Ringel \cite{ringel2013auslander,ringel2008first} proved it for artin algebras.  Bautista \cite{bautista1985algebras} and Bongartz \cite{bongartz1985indecomposables}  proved the second conjecture under the extra assumption that $k$ is algebraically closed; see also \cite{ringel1980report}. Auslander \cite{auslander1976applications} proved a theorem that is stronger than the first Brauer-Thrall conjecture and weaker than the second Brauer-Thrall conjecture. He called it Brauer-Thrall $1\frac{1}{2}$. It asserts that  a finite-dimensional algebra having infinitely many non-isomorphic indecomposable modules of some fixed $k$-dimension is actually of strongly unbounded representation type. Smal\o~\cite{smalo1980inductive} proved this theorem for artin algebras using almost split sequences and the Harada-Sai lemma.

%There is another  old  conjecture,``Pure-semisimple conjecture", which has  tight  connection to  Brauer-Thrall conjectures.
%Pure-semisimple conjecture asserts that,
%every left pure-semisimple ring (a ring which every left module is a direct sum of
%finitely generated ones) is of finite representation type.
%Auslander \cite{auslander1976large} and Ringel-Tachikawa \cite{ringel1975qf}, have proved that an artin algebra $\Lambda$ is of
%finite representation type if and only if every left $\Lambda$-module is a direct sum of finitely generated
%modules.
%The first efforts in studying the representation-theoretic properties of  lattices over (generalized)orders goes back to a seri of papers  by  Auslander \cite{auslander1976applications,auslander1985survey,auslander1976applications}, where   they  have  shown that
%right and left almost split morphisms and almost split sequences do indeed exist in the category of lattices over orders.
%Concerning the pure-semisimple conjecture, in the setting of lattices over orders,  Butler, Campbell and Kovacs \cite{butler2004infinite}, proved that, for an order $\Lambda$ over a discrete valuation domain $R$,  each large lattice ( i.e., $\Lambda$-module which is projective as an $R$-module)
%is direct sum of lattices whenever $\Lambda$ is of finite lattice type and after that Rump \cite{rump2005large}, proved the converse.\\
%i.e., he has shown that if every generalized lattice decomposes into lattice then $\Lambda$ is finite lattice type.

One of the classical situations where representation theory of rings has been studied most extensively is lattices over orders; see for example \cite{auslander1978functors} and  \cite{auslander1985survey}. Motivated by the Brauer-Thrall conjectures for finite-dimensional algebras, in this paper we investigate boundedness and strong unboundedness for lattices over orders. To explain our results in more detail, we first establish notation. Let $(R,\m)$ be a $d$-dimensional commutative noetherian complete Cohen-Macaulay local ring with a canonical module 
$\omega$ and let $\Lambda$ be an $R$-order (see Definition~\ref{D:order} for details). The category of all $\Lambda$-lattices will be denoted by 
$\Lambda$-$\textbf{lat}$ or simply $\textbf{lat}$. Denote by $S$ the direct sum of all simple $\Lambda$-modules and let $\alpha:G\lra S$ be a minimal $\textbf{lat}$-approximation (the existence of such an approximation is shown in Proposition~\ref{P:approximation}). For a given $\Lambda$-module~$M$, we set 
\[
\underline{\h}(M):=\underline{\Hom}_{\Lambda}(M,M\oplus G)
\]

We say that $M$ has finite $\underline{\h}$-length if $l_R(\underline{\h}(M))<\infty$. Notice that lattices always have finite $\underline{\h}$-length.
The order $\Lambda$  is said to be of finite $\textbf{lat}$-type if there are only finitely many non-isomorphic indecomposable lattices, and it is called of infinite $\textbf{lat}$-type if $\Lambda$ is not of finite $\textbf{lat}$-type.
Using the $\underline{\h}$-length to measure the complexity of
$\Lambda$-$\textbf{lat}$, we define the notions of bounded lattice type and strongly unbounded lattice type as usual; see Definitions \ref{d:strongly} and \ref{def1}.
\medskip

Now we summarize our results.
\begin{itemize}
 \item[1)]  An $R$-order  $ \Lambda$ is of bounded lattice type if and only if it is of finite $\textbf{lat}$-type.
 \medskip
 
 \item[2)]  If $\Lambda$-$\textbf{lat}$ contains infinitely many non-isomorphic indecomposable objects with the same $\underline{\h}$-length, then $\Lambda$ has strongly unbounded lattice type.

\end{itemize}
\medskip
%We should remind that,  motivated  by Auslander's results, studying commutative noetherian local rings and Gorenstein artin algebras
 %in connection
%with the decomposition of  maximal Cohen-Macaulay modules  have been subject of several expositions
%(see; \cite{beligiannis2011algebras, chen2008auslander, rump2005category, rump2005lattice}) see \ref{main the} for more details.
It should be noted that if $\Lambda$-$\textbf{lat}$ contains uncountably many non-isomorphic indecomposable objects, then it has strongly unbounded 
$\textbf{lat}$-type. Moreover, $R$ being an $R$-order, gives rise to the equality $R$-$\textbf{lat}$ = mCM$_0$ where mCM$_0$ is the category of  all maximal Cohen-Macaulay modules that are free on the punctured spectrum of $R$. Now assume that $R$ is a one-dimensional complete  Cohen-Macaulay local ring containing an infinite field $k$ with multiplicity $e(R)> 2$.  A slight modification of the method used in the proof 
of~\cite[Theorem 2.5]{leuschke2013brauer}, shows that  $R$ has strongly unbounded \textbf{lat}-type, see Theorem~\ref{22}. Furthermore, at the end of Section 6, examples of orders having strongly unbounded \textbf{lat}-type are presented.

 In \cite{buchweitz1987cohen} and \cite{knorrer1987cohen}  Kn\"{o}rrer and  Buchweitz-Greuel-Schreyer, showed that the local ring of an isolated simple hypersurface singularity is of finite mCM type (i.e., the number of pairwise nonisomorphic indecomposable mCM modules is finite). More generally, they showed that the ring $R=k[[x_0,...,x_d]]/(f)$, where $k$ is a field of characteristic different from 2 and  $f$ is a nonzero and non-unit element of  $k[[x_0,...,x_d]]$,  is of  finite mCM type if and only if $R \simeq k[[x_0, \cdots x_d]]/(g+x_2^2+ \cdots + x_d^2)$ for some $g\in k[[x_0, x_1]]$ such that $k[[x_0, x_1]]/(g)$ is of  finite mCM type. In the final section of this paper, we  investigate an analog of that result for bounded and strongly unbounded lattice type. We show that $R$ has strongly unbounded lattice type if and only if the double branched cover $R^{\sharp}$ of $R$ has strongly unbounded  $\textbf{lat}$-type; see Theorem~\ref{strongly}. This result leads us to deduce that $R$ has  strongly unbounded lattice type whenever $e(R) \geq 3$ and $k$ is infinite; see Theorem~\ref{21}.

We should also mention that there have been attempts to use multiplicity to establish Brauer-Thrall type results for maximal Cohen-Macaulay modules over Cohen-Macaulay rings, as well as to obtain analogs for the results of Kn\"{o}rrer and Buchweitz-Greuel-Schreyer. For more details, see Remark \ref{30}.

% There are valuable efforts that have been made to study Brauer-Thrall theorems over  Cohen-Macaulay rings and, also,  analog, for bounded CM type, of  ,  result, concerning the  multiplicity as an invariant over maximal Cohen-Macaulay modules;
%see 
%
%$R$ is a $d$-dimensional commutative complete noetherian Cohen-Macaulay local ring with maximal ideal $\m$ and a canonical module $\omega$. $\Lambda$ denotes a noetherian $R$-algebra. 
%
Throughout this paper, by a module we will always mean a left module unless stated otherwise. For a ring $\Lambda$, the category of all (respectively, finitely generated) $\Lambda$-modules will be denoted by $\Md\Lambda$ (respectively, $\md\Lambda$).

\section{Generalities}

In this section we collect, for the convenience of the reader, the basic results and definitions that  will be needed later. All of them are known, at least, to the experts, and there are no new proofs here. The only novelty is a marginally more general definition of order. The orders we are about to define belong to the wider class of noetherian algebras, a notion introduced by Auslander~\cite[p. 49]{auslander1978functors}. Let $R$ be a commutative ring and $\Lambda$ an associative ring with identity.  $\Lambda$ is an $R$-algebra means that there is a ring homomorphism $R \lra \Lambda$ whose image is in the center of $\Lambda$. One says that~$\Lambda$ is a \textit{noetherian $R$-algebra} if $R$ is noetherian and $\Lambda$ is finitely generated when viewed as an $R$-module. Henceforth, all algebras will be assumed noetherian.

First, we look at the case when $R$ is a noetherian commutative local ring and 
$\Lambda$ is free when viewed as an $R$-module.
\begin{lemma}\label{L:op}
Under the above assumptions, the following conditions are equivalent:
\begin{itemize}
 \item[1)] $\Hom_{R}(\Lambda, R)$ is a projective $\Lambda^{op}$-module,
 
 \item[2)] $\Hom_{R}(\Lambda^{op}, R)$ is a projective $\Lambda$-module.
\end{itemize}
\end{lemma}

\begin{proof}
 This is~\cite[Lemma (5.1)]{auslander1986isolated}.
\end{proof}

Of special importance to us is the case when $R$ is a \textit{complete noetherian commutative local ring} and $\Lambda$ is any noetherian $R$-algebra. Under these assumptions, we have 

\begin{lemma}\label{L:KRS}
 The Krull-Remak-Schmidt theorem holds for finitely generated $\Lambda$-modules.
\end{lemma}

\begin{proof}
 See, for example, \cite[(6.12)]{CR-1}
\end{proof}

Furthermore, under the same assumptions, we have

\begin{lemma}\label{L:semiperfect}
 A noetherian $R$-algebra is semiperfect.
\end{lemma}

\begin{proof}
 See, for example, \cite[p. 132]{CR-1}
\end{proof}

For the rest of this paper, we impose further conditions on $R$ by assuming that $(R,\m)$ is a $d$-dimensional noetherian commutative complete Cohen-Macaulay local ring with a canonical module $\omega$. Following \cite{auslander1978functors} and \cite{auslander1985survey}, we introduce
  
\begin{definition}\label{lattice-order}
A (left) $\Lambda$-module $M$ is a lattice if:

\begin{itemize}
 \item[a)] $M$ is finitely generated;
 
 \item[b)] $M$, viewed as an $R$-module, is maximal Cohen-Macaulay;
 
 \item[c)] $M_{\p}$ is $\Lambda_{\p}$-projective for each non-maximal prime ideal $\p$ of $R$;
 
 \item[d)] $\Hom_R(M,\omega)_{\p}$  is $\Lambda_{\p}^{\op}$-projective for each non-maximal prime ideal $\p$ of $R$. 
 
 \end{itemize}
 \end{definition}

This definition is similar to the one in \cite[Section 7]{auslander1978functors}. Indeed, we have modified the requirement that $R$ be an equidimensional Gorenstein ring by specializing to local rings. On the other hand, we have relaxed the Gorenstein condition to include Cohen-Macaulay rings with canonical modules. Accordingly, the canonical module $\omega$ replaces $R$ in our setting.

\begin{definition}\label{D:order}
 The $R$-algebra $\Lambda$ is called an $R$-order if $\Lambda$ is a lattice (in particular, $\Lambda$ is finitely generated as an $R$-module).
\end{definition}

\begin{lemma}\label{L:opposite}
If $\Lambda$ is an $R$-order, then so is $\Lambda^{\op}$.
\end{lemma}

\begin{proof}
The argument below is similar to that in the proof of Lemma~\ref{L:op}.
Assume that $\Lambda_{\p} \simeq P_1^{k_{1}}\oplus\cdots\oplus P_n^{k_{n}}$, where the $P_i$ form a complete set of
non-isomorphic indecomposable projective $\Lambda_{\p}$-modules. By assumption, $\Hom_R(_{\Lambda}{\Lambda},\omega)_{\p}$ is
$\Lambda_{\p}^{\op}$-projective, and so $\Hom_{R_{\p}}(P_i,\omega_{\p})$, $1\leq i\leq n$ are pairwise non-isomorphic indecomposable projective 
$\Lambda_{\p}^{\op}$-modules. Since $\Lambda$ is semiperfect, a complete set of non-isomorphic indecomposable projective $\Lambda_{\p}^{\op}$-module has precisely $n$ elements. Therefore, $\Hom_{R_{\p}}(P_i,\omega_{\p})$, $1\leq i\leq n$ is such a complete set and 
\[
\Lambda_{\p}^{\op} \simeq \oplus_{i=1}^n \Hom_{R_{\p}}(P_i, \omega_{\p})^{m_{i}}
\]
for some integeres $m_{i}$. By dualizing with 
$\Hom_{R_{\p}}(-,\omega)_{\p}$, we have that $\Hom_R(\Lambda^{\op},\omega)_{\p} \simeq \oplus_{i=1}^n P_i^{m_{i}}$ is ${\Lambda}_{\p}$-projective. \end{proof}

The following result is obvious.

\begin{lemma}
 The functor $(-)^*=\Hom_R(-,\omega)$ is a duality  between 
$\Lambda$-$\textbf{lat}$ and $\Lambda^{\op}$-$\textbf{lat}$. \qed
\end{lemma}

\smallskip

\section{Lattice approximations}

Let $\Lambda$ be an $R$-order as above and $\Lambda$-$\textbf{lat}$ the category of all lattices over $\Lambda$. The full subcategory of 
$\Lambda$-$\textbf{lat}$ determined by indecomposable lattices will be denoted by $\ind(\Lambda$-$\textbf{lat})$. In this section we establish the existence of a minimal $\textbf{lat}$-approximation for $\Lambda$-modules of finite length and their syzygy modules.

%2) Let $M$ be a finitely generated $R$-module. A sequence of elements $\xx=x_1, \cdots, x_n\in\m$ is called an $M$-regular sequence,
%provided that $x_i$ is a nonzerodivisor on $M/(x_1, \cdots, x_{i-1})M$ for any $1\leq i\leq n$ (for $i=1$, we mean
%that $x_1$ is a nonzerodivisor on $M$) and $(x_1, \cdots, x_n)M\neq M$.
%$R$-module $M$ is called maximal Cohen-Macaulay, if there is an $M$-regular sequence
%$\xx=x_1, \cdots, x_d$ where $d=\di R$. The category of all maximal Cohen-Macaulay $R$-modules will be denoted by $\mcm$.
%
%
%
%Let $\Lambda$ be a noetherian $R$-algebra.
%A $\Lambda$-module $M$ is said to be a \texttt{lattice}, if
%$M$ is a maximal Cohen-Macaulay $R$-module and
%$M_{\p}$ and $\Hom_R(M,\omega)_{\p}$ are $\Lambda_{\p}$-projective and
%$\Lambda_{\p}^{\op}$-projective, respectively for all non-maximal prime ideals
%$\p$ of $R$. A finitely generated $R$-algebra $\Lambda$ is called an $R$-order if it is a lattice.
%Obviously, $R$ is an $R$-order, whenever
%$R_{\p}$ is a Gorenstein ring for any non-maximal prime ideal $\p$ of $R$.

First, recall some basic definitions. A homomorphism $f:M \lra N$ is said to be \texttt{right minimal} if every homomorphism $g:M \lra M$ satisfying $fg = f$ is an isomorphism.

\begin{definition}
Let $\X$ be a subcategory of $\md\Lambda$.

\begin{enumerate}
 \item $Y \in \mathcal{X}$ is said to be $\mathcal{X}$-\texttt{injective}, if $\Ext^{1}_{\Lambda}(M,Y) = 0$ for any $M \in \mathcal{X}$.
 \smallskip
 
 \item Let $M$ be a $\Lambda$-module. An exact sequence of $\Lambda$-modules
\[
0 \lra Y_{M} \lra X_{M} \overset{f}\lra M \lra 0
\]
is an $\mathcal{X}$-approximation of $M$, if $X_{M} \in \mathcal{X}$ and $Y_{M}$ has a finite resolution
\[
0 \longrightarrow L_{t} \longrightarrow \ldots \longrightarrow L_{0}
\longrightarrow Y_{M} \longrightarrow 0
\]
by $\mathcal{X}$-injectives. The above approximation is said to be \texttt{minimal} if $f$ is right minimal.
\end{enumerate}

\end{definition}

\begin{remark}
If $Z \in \X$, then any $g : Z \lra M$ lifts over the $\X$-approximation 
$X_M\st{f}\lra M$.
\end{remark}

Given a $\Lambda$-module $M$, the full subcategory of $\md\Lambda$ consisting of all modules isomorphic to direct summands of finite sums of copies of $M$ will be denoted by $\add\,M$.

\begin{lemma}\label{L:CM-inj}
The category of $\textbf{lat}$-injective modules coincides with $\add\, \Hom_{R}(\Lambda_{\Lambda}, \omega)$.
\end{lemma}

\begin{proof}
 First we show that $\Hom_{R}(\Lambda_{\Lambda}, \omega)$ is $\textbf{lat}$-injective. Picking an arbitrary $M\in\Lambda$-$\textbf{lat}$ and using the 
Hom - tensor adjunction, we have an isomorphism
 \[
 \Ext^{i>0}_{\Lambda}(M, \Hom_{R}(\Lambda_{\Lambda}, \omega)) \simeq \Ext_R^{i>0}(M, \omega) = 0.
 \]
 Conversely, assume that $Y$ is a $\textbf{lat}$-injective. Dualizing into $\omega$ a syzygy sequence
\[
0 \longrightarrow Z \longrightarrow \Lambda^{n}_{\Lambda} \longrightarrow \Hom_R (Y, \omega) \longrightarrow 0,
\]
all of whose terms are $\Lambda^{op}$-lattices, we have an exact sequence
\[
0 \longrightarrow Y \longrightarrow \Hom_{R}(\Lambda^{n}_{\Lambda}, \omega)   \longrightarrow  \Hom_R(Z, \omega) \longrightarrow 0.
\]
Since $\Hom_R(Z, \omega)$ is in $\Lambda$-$\textbf{lat}$ and $Y$ is
$\textbf{lat}$-injective, this sequence splits, proving the claim.
\end{proof}

\begin{definition}
 We say that $\Lambda$-$\textbf{lat}$ has enough $\textbf{lat}$-injectives if any lattice can be embedded in a $\textbf{lat}$-injective with cokernel in 
 $\Lambda$-$\textbf{lat}$.
\end{definition}

As a consequence of the proof of Lemma~\ref{L:CM-inj}, we have
\begin{cor}\label{C:enough}
 For any $R$-order $\Lambda$, the category $\Lambda$-$\textbf{lat}$ has enough $\textbf{lat}$-injectives. \qed
\end{cor}

 Of special interest to us are approximations by $\Lambda$-$\textbf{lat}$. 
 %In particular, we have the next result.
\begin{prop}\label{P:approximation}
Any $\Lambda$-module of finite length and any of its syzygy modules have  minimal $\textbf{lat}$-approximations.
\end{prop}

\begin{proof}
 The proof repeats the Auslander-Buchweitz construction of $\mcm$-approximation and goes as follows. If $C$ is a $\Lambda$-module of finite length, then $C$ is of finite length over $R$ and hence a Cohen-Macaulay $R$-module of codepth $d$. Let $C^{\vee} := \Ext^{d}_{R}(C, \omega)$ be the Cohen-Macaulay dual of $C$; this is a $\Lambda_{\Lambda}$-module. Choosing a   $\Lambda_{\Lambda}$-projective resolution
 \[
 \ldots \longrightarrow P_{1} \overset{\partial_{1}}{\longrightarrow} P_{0} \longrightarrow  C^{\vee} \longrightarrow 0,
 \]
and dualizing it into $\omega$, we have a complex of $\Lambda$-modules;
\[
0 \longrightarrow P_{0}^{\ast} \overset{\partial_{1}^{\ast}}\longrightarrow P_{1}^{\ast} \longrightarrow \ldots \longrightarrow
P^{\ast}_{d-1} \overset{\partial_{d}^{\ast}}\longrightarrow
P^{\ast}_{d} \overset{\partial_{d+1}^{\ast}}\longrightarrow
P^{\ast}_{d+1}\longrightarrow \ldots
\]
whose only homology, $C^{\vee \vee}$, is concentrated in degree $d$. Thus we have a short exact sequence
 \[
 0 \longrightarrow \im\, \partial_{d}^{\ast} \longrightarrow  \Ker\, \partial_{d+1}^{\ast} \longrightarrow C^{\vee \vee} \longrightarrow 0
 \]
 of $\Lambda$-modules. By the duality for Cohen-Macaualy modules, the  ${P_{i}^{\ast}}$ are in $\mcm$, implying that $\Ker\, \partial_{d+1}^{\ast}$ is in $\mcm$ as well. That this module localizes to projectives at non-maximal prime ideals follows from the fact that $C$ is of finite length and, therefore, the localization of its projective resolution at non-maximal prime ideals yields a split exact complex of projectives. As~$C$ is a Cohen-Macaulay $R$-module, we have a canonical isomorphism $C \to C^{\vee\vee}$ of $R$-modules, which is also a 
 $\Lambda$-homomorphism.  By Lemma~\ref{L:CM-inj}, each $P_{i}^{\ast}$ is 
 $\textbf{lat}$-injective. Therefore, $\im\, \partial_{d}^{\ast}$ has a finite resolution by $\textbf{lat}$-injectives, and as a result we have a $\textbf{lat}$-approximation of~$C$.

To construct a $\textbf{lat}$-approximation for the first syzygy module of $C$, start with a $\textbf{lat}$-approximation $0 \lra Y \lra X \lra C \lra 0$ and a syzygy sequence $0 \lra \Omega C \lra P \lra C \lra 0$ of $C$, together with a short exact sequence $0 \lra Y_0 \lra L_{0} \lra Y \lra 0$, where $L_{0}$ is the first term in the finite resolution of $Y$ by $\textbf{lat}$-injectives. By the horseshoe lemma, we have a commutative diagram with exact rows and columns:

\[
\xymatrix
	{
	& 0 \ar[d]
	& 0 \ar[d]
	& 0 \ar[d]
	&
\\
	0 \ar[r]
	& Y_0 \ar[r] \ar[d]
	& X' \ar[r] \ar[d]
	& \Omega C \ar[r] \ar[d]
	& 0
\\
	0 \ar[r]
	& L_{0} \ar[r] \ar[d]
	& P \oplus L_{0} \ar[r] \ar[d]
	& P \ar[r] \ar[d]
	& 0
\\
	0 \ar[r]
	& Y \ar[r] \ar[d]
	& X \ar[r] \ar[d]
	& C \ar[r] \ar[d]
	& 0
\\
	& 0
	& 0
	& 0
	&
	}
\]
Notice that $X'$, being the kernel of a lattice homomorphism, is a lattice itself, which yields a $\textbf{lat}$-approximation of $\Omega C$.

It remains to show the existence of minimal $\textbf{lat}$-approximations. This is done in the next lemma, which should be known at least to the experts.
\end{proof}

\begin{lemma}\label{minimal}
 If a $\Lambda$-module has a $\textbf{lat}$-approximation, then it has a minimal one.
\end{lemma}

\begin{proof}
 Let $f : X \lra M$ be a $\textbf{lat}$-approximation. If $f$ is right minimal, we are done. Thus assume that $f = f g$ for some non-invertible endomorphism of $X$. Let the lower index $i$ indicate, for both objects and maps, reduction modulo $\m^{i}$. If  $g_{i}$ is surjective, then by Nakayama's lemma, the same would be true for $g$, and since $X$ is noetherian, $g$ would be an isomorphism. Thus no $g_{i}$ is an isomorphism. As $X_{i}$ is a module of finite length, by applying the Fitting lemma, we have a direct sum decomposition $X_{i} \simeq X_{i}' \oplus X_{i}''$ into $g_{i}$-stable submodules, where the restriction of  $g_{i}$ to $X_{i}'$ is nilpotent and the restriction of  $g_{i}$ to $X_{i}''$ is an isomorphism. $X_{i}'$ is uniquely determined as the largest submodule on which $g_{i}$ is nilpotent. By the foregoing argument, it is nonzero. Since $f = fg = fg^{2} = \ldots$, we have that $f_{i}$ vanishes on $X_{i}'$.

 For each $i$, let $\pi_{i} : X_{i} \to X_{i-1}$ denote reduction modulo $\m^{i-1}$. This is a surjective map. Moreover, $\pi_{i} (X_{i}') \subset X_{i-1}'$ and  $\pi_{i} (X_{i}'') \cap  X_{i-1}' = \{0\}$. It follows that  $\pi_{i} : X_{i}' \to X_{i-1}'$ is surjective, too. Let $p_{i} : X_{i} \to X_{i}$ be the projection onto $X_{i}'$. Then  $p_{i}^{2} = p_{i}$ and we set $p : = \underleftarrow{\lim}\, p_{i}$. This is an idempotent endomorphism of the completion of $X$, which is isomorphic to $X$, because $X$ is finitely generated over $R$ and $R$ is complete. Moreover, since the restriction of each $\pi_{i}$ to $X_{i}'$ is surjective, one concludes that $p$ is a non-trivial idempotent and therefore the image $X'$ of $p$ is a nonzero direct summand of $X$. Clearly, $f$ vanishes on that direct summand. In other words, the approximation has a common nonzero direct summand with its kernel. Removing this common direct summand, we have a right minimal approximation, as needed.
\end{proof}

As a consequence of the proof of Proposition~\ref{P:approximation}, we have 

\begin{cor}\label{C:fpd}
 Under the assumption of Proposition~\ref{P:approximation}, suppose that $R$ is Gorenstein and $\Lambda = R$. Then the minimal lattice approximation of a module of finite length is also its minimal mCM approximation. In particular, the kernel of the minimal approximation is of finite projective dimension. \qed
\end{cor}

\begin{proof}
 The proof of Proposition~\ref{P:approximation} shows that the construction of a lattice approximation for a module of finite length is the same as that of an mCM approximation. Its kernel is a module of finite injective dimension, but since $R$ is Gorenstein, it is also of finite projective dimension.
\end{proof}

\begin{remark}
%Since $\Lambda$ is a finitely generated algebra over a complete noetherian
%commutative local ring, the Krull-Remak-Schmidt theorem holds for $\md\Lambda$ (see, for example, \cite[Exercise 6, p. 88]{reiner1975maximal}).
%As a consequence, 
By Lemma~\ref{L:KRS}, we have a Krull-Remak-Schmidt theorem for lattices over $R$-orders, i.e., each lattice is a direct sum of indecomposable lattices
and this decomposition is unique up to isomorphism. Also, by Lemma~\ref{L:semiperfect}, $\Lambda$ is semiperfect, i.e., each finitely generated $\Lambda$-module has a projective cover.
%
%
%We also recall that $R$-orders are semiperfect \cite[Proposition 4.1]{krause2015krull}, i.e., each finitely generated $\Lambda$-module has a projective cover.
\end{remark}

Suppose that $\{S_1, \cdots, S_n\}$ is a complete set of non-isomorphic simple $\Lambda$-modules, and let $S := \oplus_{i=1}^nS_i$. By Proposition~\ref{P:approximation} and Lemma \ref{minimal}, there exists a minimal $\textbf{lat}$-approximation $\alpha:G\lra S$. Now we introduce the notion of $\underline{\h}$-length, an invariant to measure the size of lattices.

\begin{definition}
Let $M$ be a (not necessarily finitely generated) $\Lambda$-module. We set 
\[
\underline{\h}(M) := \underline{\Hom}_{\Lambda}(M,M\oplus G)
\]
and define the $\underline{\h}$-\texttt{length} of $M$ as $l_R(\underline{\h}(M))$, where $G$ is a minimal $\textbf{lat}$-approximation of $S$ and $l_R$ denotes the length over $R$. 
\end{definition}

Recall that, for any $\Lambda$-modules $M$ and $N$,
\[
\underline{\Hom}_{\Lambda}(M, N)=\Hom_{\Lambda}(M, N)/{P(M, N)},
\]
where $P(M, N)$ is the $R$-submodule of $\Hom_{\Lambda}(M, N)$ consisting of all homomorphisms factoring through projective $\Lambda$-modules. It follows from the definitions that $l_R(\underline{\h}(M))$ is finite for any lattice~$M$.

\begin{definition}
 Let $\mathcal{L}$ be the subcategory of $\Lambda$-$\textbf{lat}$. We say that 
 $\mathcal{L}$ is of \texttt{bounded $\underline{\h}$-length}, if there is an integer $b>0$ such that $|\, \underline{\h}(\mathcal{L}) |= \sup\, \{ l_R(\underline{\h}(M))\, |\, M\in\mathcal{L}\} < b$.
\end{definition}

\section{The Harada-Sai lemma and the stable length}
%$\underline{h}$-length

In this section we study relationships between the $\underline{\h}$-length and the betti numbers of lattices. This leads to a variant of the Harada-Sai lemma for lattices based on the $\underline{\h}$-length.
\smallskip

Recall that a homomorphism $f:M\lra N$ of $\Lambda$-lattices is said to be  \texttt{irreducible} if it is neither a section nor a retraction, and in any factorization
\[
\xymatrix{ M \ar[dr]_g \ar[rr]^{f} && N \\ & X \ar[ur]_h & }
\]
with $X$ in $\Lambda$-$\textbf{lat}$, either $g$ is a section or $h$ is a retraction.

\begin{definition}
Let $M$ and $N$ be indecomposable $\Lambda$-lattices and $\xx=x_1, x_2, \cdots, x_n$ an $R$-regular sequence.
%objects in $\ind\mathcal{L}$.
We say that $N$ \texttt{is connected} to~$M$ modulo $\xx$, if $N$ is isomorphic to $M$ or if there is a chain
\[
N=L_0\st{f_1}\lra L_1\lra\cdots L_{n-1}\st{f_n}\lra L_n=M
\]
of irreducible maps between indecomposable lattices $L_{i}$ such that $f_n\cdots f_1 \otimes_{{\Lambda}}\Lambda/\xx\Lambda \neq 0$. We say that $N$ is connected to $M$ if there is an $R$-regular sequence such that $N$ is connected to $M$ modulo that sequence.
\end{definition}

The relation `is connected to' is clearly reflexive, but not necessarily symmetric. It also fails to be transitive, and the classical Harada-Sai lemma for finite modules over finite-dimensional algebras provides a rough quantitative measure of this failure. Our next goal is to establish a similar result for lattices over orders.
% We recall that a functor $F$ is simple if it has no subfunctors other than itself and the zero functor. We say
%that $F$ has finite length $n$, $l(F)=n$, if there is a chain
%$ F_1 \subset  F_2 \cdots \subset  F_n = F$
%such that $F_1$ and all the quotients $ F_i / F_{i-1}$ are simple functors.

%First we want to set up notation.
%Let ${\rm fl}R$ denote the category of $R$-modules of finite length. Let $\mathcal{T}$ a subcategory of $\Lambda$-$\textbf{lat}$  and $F: \mathcal{T} \rightarrow {\rm fl}R$ an additive functor.

%\begin{definition}
%The norm, $|F({\mathcal{T}})|$, of $F$ on $\mathcal{T}$ is defined as $\sup \{ \mathrm{length}_{R}(F(N))| N\in \ind\mathcal{T} \}$, where $N$ ranges over all $N$ in $\mathcal{T}$.
%\end{definition}

%Thus the norm of $F$ on $\T$ is either a nonnegative integer or $\infty$. Of special interest to us will be the functor ``covariant Hom modulo projectives'', denoted by $\underline{h}(X) := \underline{\Hom}_{\Lambda}(X,X\oplus G)$, where $X \in\mathcal{T}$
%and $|\underline{h}(M)|:=l_R(\underline{\Hom}_{\Lambda}(M,M\oplus G))$.
%Let $\T$ be a subcategory of $\textbf{lat}$. We say that $\T$ is of bounded $\underline{h}$-length, if there
%is an integer $b>0$ such that $|\underline{h}(\T)|=$sup$\{|\underline{h}(M)|~~| M\in \ind\T\}<b$.
%Since $M$ is a lattice, $|\underline{h}(M)|$ is finite.

%In that case we use the notation $|\underline{h}^{X}(\mathcal{T})|$. It is sometimes convenient to allow $X$ vary in another subcategory $\mathcal{T}_{0}$, and in that case we shall use the notation $|\underline{h}^{\mathcal{T}_{0}}(\mathcal{T})|$.
\medskip

We begin by recalling the notion of a \texttt{faithful system of parameters} 
(faithful s.o.p.) of~$R$ for a subcategory of $\Lambda$-$\textbf{lat}$ \cite[Definition 15.7]{leuschke2012cohen}.\footnote{Such a system is also called an \texttt{efficient system of parameters} of $R$ in \cite[p. 48]{yoshino1990cohen}.}

%\begin{definition}
% Let $\mathcal{M} \subset \ind\mathcal{L}$ be a set of indecomposable lattices. We shall say that a system of parameters $\xx := x_{1}, \ldots x_{d}$ of~$R$ is
%$\mathcal{M}$-\textit{faithful} if for each $M \in \mathcal{M}$, $M/\xx M$ is an indecomposable $\Lambda/\xx\Lambda$-module and the canonical ring homomorphism
%\[
%\Hom_{\Lambda}(M,M) \longrightarrow
%\Hom_{\Lambda/\xx\Lambda}(M/\xx M,M/\xx M)
%\]
%is onto.
%\end{definition}

\begin{definition}
Let $\mathcal{L}$ be a subcategory of $\Lambda$-$\textbf{lat}$. A system of parameters $\xx = x_1, \ldots, x_d$ of~$R$ is said to be  
$\mathcal{L}$-\texttt{faithful} (or faithful s.o.p. for $\mathcal{L}$) if
$\xx\Ext_{\Lambda}^1(M, N) = 0$ for any $M$ in $\mathcal{L}$ and any finitely generated $\Lambda$-module $N$.
\end{definition}

\begin{lemma}\label{L:stable-vanishing}
 Let $M$ be a finitely generated $\Lambda$-module and $\xx = (x_1, \ldots, x_n)$ an ideal of $R$. The following conditions are equivalent:
 \begin{enumerate}
 \item\label{1} $\xx\Ext_{\Lambda}^1(M, -) = 0$ when restricted to finitely generated $\Lambda$-modules;
 \smallskip
 \item\label{2} $\xx\underline{\Hom}_{\Lambda}(M, -) = 0$ when restricted to finitely generated $\Lambda$-modules;
  \smallskip
 \item\label{3} $\xx\underline{\Hom}_{\Lambda}(M, M) = 0$;
  \smallskip
 \item\label{4} For any $x \in \xx$,
% $i = 1, \ldots, d$, 
 multiplication by $x$ on $M$ factors through a $\Lambda$-projective.
\end{enumerate}
If one of these conditions holds, then the vanishing in \eqref{1} and  \eqref{2} occurs on all $\Lambda$-modules, not just on finitely generated ones.
\end{lemma}

\begin{proof}
Since $x$ is a central element of $\Lambda$, the expressions in the first three conditions are well-defined. Also, multiplication by $x$ in~\eqref{4} is a homomorphism of $\Lambda$-modules.

$\eqref{1} \Rightarrow \eqref{2}$. For a finitely generated $\Lambda$-module $X$, choose a short exact sequence $0 \to \Omega X \to P \to X \to 0$ with $P$ a finitely generated projective. Then $\Omega X$ is finitely generated, too. Applying $\Hom(M, -)$ and passing to the corresponding long homology exact sequence, we have that ${\underline{\Hom}_{\Lambda}(M, X)}$ is an $R$-submodule of $\Ext_{\Lambda}^1(M, \Omega X)$. Since the latter is annihilated by $\xx$, the same holds for the former.

$\eqref{2} \Rightarrow \eqref{3}$. Trivial.

$\eqref{3} \Rightarrow \eqref{4}$. Multiplication by $x$ on $M$ can be written as $x1_{M}$. The desired result now follows.
%By a theorem of Hilton-Rees \cite{hilton1961natural}, $\underline{\Hom}_{\Lambda}(M,M)$ is isomorphic to the endomorphisms ring of the functor $\Ext_{\Lambda}^{1}(M,-)$. Therefore,  $\xx \Ext_{\Lambda}^1(M,- )= 0$.

$\eqref{4} \Rightarrow \eqref{1}$. Immediate.
\newline The last assertion of the lemma follows immediately from \eqref{4}.
\end{proof}

The following proposition is now obvious.

\begin{prop}\label{prop2}
Let $\mathcal{L}$ be a subcategory of $\Lambda$-$\textbf{lat}$ and $\xx = x_1, \ldots, x_d$ a system of parameters of $R$. Then $\xx$ is $\mathcal{L}$-faithful if and only if $\xx\underline{\h}(M)=0$ for any $M$ in $\mathcal{L}$. \qed
\end{prop}

\begin{lemma}\label{L:bound-sop}
Let $b>0$ be an integer. Then there is a system of parameters $\xx$ of $R$ 
which is faithful for any subcategory $\mathcal{L}$ of $\Lambda$-$\textbf{lat}$
with $|\underline{\h}(\mathcal{L})|\leq b$.
\end{lemma}

\begin{proof}
By assumption, $\m^{b} \underline{\h}(M) = 0$ for any $M \in \mathcal{L}$. Now choose an s.o.p. $\xx$ inside $\m^{b}$ and use
Proposition \ref{prop2}. 
\end{proof}
%\begin{lemma}\label{L:independence}
% For a given integer $b$, there is an s.o.p. $\xx$ which is $\mathcal{L}$-faithful for any $\mathcal{L} \subset \Lambda$-$\textbf{lat}$ with
%$| \underline{h}(\mathcal{L}) | \leq b$.
%\end{lemma}

%\begin{proof}
%Under the above assumption, $\m^{b} \underline{h}(M) = 0$ for any $M \in \mathcal{L}$. Now choose an s.o.p inside $\m^{b}$ and use
%Proposition~\ref{prop2}.
%\end{proof}

%\begin{prop}\label{200}
% Let $\xx =x_1, \ldots, x_d$ be a $\mathcal{L}$-faithful s.o.p., and $M$ and $N$ lattices in ??.

%\begin{enumerate}

%\item Suppose that $f:M/\xx^2M\lra N/\xx^2N$ is an isomorphism. Then there exists an isomorphism
 %$g:M\lra N$ such that $f\otimes_{\Lambda}\Lambda/\xx\simeq g\otimes_{\Lambda}\Lambda/\xx$,
 %where $\xx^2=x_1^2, \ldots, x_d^2$.

%\item $M$ is indecomposable if and only if $M/\xx^2M$ is indecomposable. %as $\Lambda$-module.

%\end{enumerate}
%\end{prop}

\begin{definition}
 If $M$ is an object in $\Lambda$-$\textbf{lat}$, the number of indecomposable direct summands (with multiplicities) which appear in the projective cover of~$M$ will be denoted by $\beta (M)$.

\end{definition}

For any lattice $M$, we have the numerical invariants $\beta(M)$, $e(M)$ (multiplicity of $M$ as an $R$-module), and $\beta_R(M)$ (the betti number of $M$ as an $R$-module). Let $\mathcal{L}$ be a subcategory of $\Lambda$-$\textbf{lat}$ and $\gamma$ one of the functions $\beta, e, \beta_R$. We say that $\gamma (\mathcal{L})$ is \texttt{bounded} if there is an integer $b>0$ such that 
$\gamma (M)\leq b$ for any $M\in\mathcal{L}$.

\begin{lemma}\label{L:betti}
If $\mathcal{L}\subseteq\ind(\Lambda$-$\textbf{lat})$ is such that 
$|\underline{\h}(\mathcal{L})|\leq b$ for some integer $b$, then 
$\beta(\mathcal{L}) \leq b$.
\end{lemma}

\begin{proof}
Since there are only finitely many indecomposable projective $\Lambda$-modules, we may assume that $\mathcal{L}$ does not contain projectives. 
%%%%%By~\cite[Proposition 4.1]{krause2015krull},
%\fbox{Need a better reference} \marginpar{$\bigotimes$}
%Since $\Lambda$ is semiperfect, the modules
%$S_{i} : = P_{i}/ \ra P_{i}$, $i = 1, \ldots , n$, where the $P_{i}$ are the principal projectives,  form a complete set of simple $\Lambda$-modules and the canonical maps $\phi_{i} : P_{i} \lra S_i$ are projective covers. 
Returning to 
$S = \oplus_{i=1}^n S_i$, 
%\marginpar{\fbox{w/ multiplicities?}} 
where the $S_{i}$ form a complete set of simple modules, we have that, for any $L\in\mathcal{L}$, the canonical map
$\Hom_{\Lambda}(L,S) \longrightarrow \underline{\Hom}_{\Lambda}(L,S)$ is an isomorphism.  As a consequence,
\[
l_R(\Hom_{\Lambda}(L,S)) = l_R(\underline{\Hom}_{\Lambda}(L,S)) \leq b.
\]

%======================
%and the minimal 
%$\textbf{lat}$-approximation $\alpha:G\lra S$, we have that 
%======================
%If $\alpha : G \lra S $ is a \textcolor{red}{right minimal $\textbf{lat}$-approximation}, then $\underline{\Hom}_{\Lambda}(M,G )\simeq\underline{\Hom}(M,\oplus_{i=1}^sG_i)$, where the $G_{i} \lra S_{i}$ are \textcolor{red}{right minimal $\textbf{lat}$-approximation} of the simple summands.
%=========================
%Since $\alpha$ is an approximation, the map 
%\[
%\underline{\Hom}_{\Lambda}(M,\alpha):\underline{\Hom}_{\Lambda}(M,G)\lra\underline{\Hom}_{\Lambda}(M,S)
%\]
%is epic. 
%========================== 

%\fbox{Need a better reference}  \marginpar{$\bigotimes$}
Let $f:P \to L$ be a projective cover and $P \simeq P_1\oplus\cdots\oplus P_t$, where the $P_i$ are indecomposable projective $\Lambda$-modules. Thus
$t = \beta(L)$, the betti number of $L$. Writing $f = [f_1 \cdots f_t]$, we have that $\im f_i\neq 0$ and $\Sigma_{i\neq j} \im f_j$ does not contain $\im f_i$ 
for any $1\leq i\leq t$. Let $L_i:= L / I_i$, where $I_i=\Sigma_{i\neq j} \im f_j$. Since $L_i\neq 0$, we have $\Hom_{\Lambda}(L_i,S_j)\neq 0$ for some $j$, which we denote by $i$. Pick a nonzero $g_i \in \Hom_{\Lambda}(L_i,S_i)$ and set $\psi_i := \lambda_i g_i \pi_i$, where $\pi_i : L \to L_i$  is the canonical projection and  $\lambda_i : S_i \to S$ is the canonical injection. By definition, 
$\psi_if_i\neq 0$ and $\psi_jf_i=0$ for $1\leq i,j\leq t$ and $i\neq j$. Applying 
$\Hom_{\Lambda}(-,S)$ to the projective cover $f$ of $L$, we have a monic $R$-homomorphism 
\[
\Hom_{\Lambda}(f,S):\Hom_{\Lambda}(L,S)\lra \Hom_{\Lambda}(P_1\oplus\cdots\oplus P_t,S).
\] 
Let $\Psi_{i}$ be the image $\psi_{i}f$ of $\psi_{i}$, $1 \leq i \leq t$. The above orthogonality relations show that $\Psi_{i} = [0, \ldots, \psi_{i}f_{i}, \ldots, 0]$. The image of $\Psi_{i}$ is isomorphic to $S_{i}$ because $S_{i}$ is simple and 
$\psi_{i}f_{i} \neq 0$. Moreover, that image is a direct summand of $S$. Viewing $S$ as an $R$-module and reducing modulo the maximal ideal of $R$, we have that the reductions of the $\Psi_{i}$ remain nonzero. It is now easy to see that these reductions are linearly independent over the residue field of $R$. By Nakayama's lemma (over $R$), the $\Psi_{i}$ form a generating set for the image of 
$\Hom_{\Lambda}(f,S)$, which is isomorphic to $\Hom_{\Lambda}(L,S)$. Hence
\[
\beta(L) = t \leq l_R(\Hom_{\Lambda}(L,S)) \leq b.
\]
%
%For any element $\psi_i\in \Hom_{\Lambda}(M,S)$, set $Z_i=\Hom_{\Lambda}(f,S)\psi_i\in \im \Hom_{\Lambda}(f,S)$. We have  that $Z_i\neq 0$ for any $1\leq i\leq t$. By the definition of the $M_i$, $RZ_1\oplus\cdots\oplus RZ_t$ is a submodule of $\im \Hom_{\Lambda}(f,S)\simeq \Hom_{\Lambda}(M,S)$. So $l_R(RZ_1\oplus\cdots\oplus RZ_t)\leq b$ and $\beta(M)\leq b$.
\end{proof}

Next, we want to recall, without  proof, the original Harada-Sai lemma \cite[Lemma 9, p. 336]{harada1970categories}.

\begin{lemma}\label{hslemma}
Let $\Lambda$ be any ring and $\{M_i\}_{i \in \mathbb{N}}$ a family of indecomposable modules of length less than or equal to $b$. Suppose that, for all $i$, we are given non-isomorphisms $f_i:M_i\lra M_{i+1}$.  Then there is an integer $n>0$ such that $f_n \cdots f_1=0$.
\end{lemma}

The following is a variation of the Harada-Sai lemma based on the betti numbers, which is also known.
\begin{lemma}[A Harada-Sai lemma for lattices based on betti numbers]
\label{L:Harada-Sai}
Let $\mathcal{L}$ be a subcategory of $\Lambda$-$\textbf{lat}$ with a bounded
$\beta(\mathcal{L})$. If $\xx$ is a faithful system of parameters for $\mathcal{L}$,
then there is $r \in \mathbb{N}$ such that any $r$-fold composition of non-isomorphisms in $\ind(\mathcal{L})$ is zero modulo $\xx^2$.
\end{lemma}

\begin{proof}
As $\xx$ is $\mathcal{L}$-faithful, $M_j/\xx^2 M_j$ is indecomposable for any indecomposable object $M_j \in\mathcal{L}$ because of
 \cite[Corollary 15.11]{leuschke2012cohen}. 
If $\oplus_{i=1}^n P_i^{t_i} \to M_j$ is a projective cover with indecomposable projectives $P_{i}$, then, by assumption,  
$\Sigma_{i=1}^nt_i\leq b$ for some $b>0$ independent of~$M_{j}$. Also,
\[
l(M_j/\xx^2 M_j) \leq \Sigma_{i=1}^{n} t_i l(P_{i}/\xx^2 P_{i}) \leq
%\Sigma_{i=1}^{b} l(\Lambda/\xx^2 \Lambda) =
bl(\Lambda/\xx^2 \Lambda).
\]
%where $\Sigma_{i=1}^nt_i\leq b$.
If $f_i:M_i\lra M_{i+1}$ are non-isomorphisms in $\ind(\mathcal{L})$, then, by Nakayama's lemma, the family 
$\{f_{i} \otimes_{\Lambda} \Lambda/\xx^2 \Lambda \}$ consists of non-isomorphisms, too. By Lemma \ref{hslemma}, we are done.
\end{proof}

\begin{lemma}[A Harada-Sai lemma for lattices based on $\underline{\h}$-length]\label{T:Harada-Sai}
%\fbox{\parbox{5.5in}{\textbf{Old version}: Assume that $\mathcal{L}$ is a subcategory of  ${\CM}$ consists of indecomposable objects of
%$\mathcal{C M }$. If there exists an integer $b$ such that
%$|h^{ \underline{M}}|\leq b$, for any $M$ of $\mathcal{L}$, then there exists an integer $r$ such that for any chain of objects $\mathcal{L}$,
%\[
%M_1\st{f_1}\lra  M_2\st{f_2}\lra \cdots M_{r-1}\st{f_{r-1}}\lra  M_{r}
%\]
%where the $f_i$ are non-isomorphisms, $f_{r-1}\cdots f_1\otimes_{\Lambda}\Lambda/\xx^2 = 0$ and $\xx$ is a faithful system of parameters.}}
Let $\mathcal{L}$ be a subcategory of $\Lambda$-$\textbf{lat}$ of bounded $\underline{\h}$-length and $\xx$  an $\mathcal{L}$-faithful s.o.p. (such a system of parameters of $R$ exists by Lemma~\ref{L:bound-sop}). Then there is a number $r\in \mathbb{N}$ such that, for any chain
\[
M_{1} \overset{f_{1}}\longrightarrow M_{2}
\overset{f_{2}}\longrightarrow M_3\st{f_3}\lra\cdots
\]
of non-isomorphisms between indecomposable objects in  $\mathcal{L}$ and any integer $j$, we have
\[
f_{r+j}  \ldots f_{j} \otimes \Lambda/\xx^2\Lambda = 0.
\]
%Let $\mathcal{T} \subset \ind \mathcal{L}$ be a class of indecomposable lattices such that  $|\underline{h}^{\mathcal{T}}(\T \cup \mathcal{L}_{0}) | \leq p$ for some integer $p$. Let
%. Then there is a $ \mathcal{T}$-faithful s.o.p. $\xx$ and a positive integer $r$ such that
%.
%
\end{lemma}

\begin{proof}

%\fbox{\parbox{5in}{\textbf{Old proof}: According to Lemma \ref{lem5}, there is a faithful system of parameters $\yy^2$ and set $\xx=\yy^2$. By Lemma \ref{L:betti}, ${\CM}$ is finitely generated. So  there exists a finite set $\X =\{X_1, \cdots, X_n\}$ such that for any $M_i\in\mathcal{L}$, there is an epimorphism  $\oplus_{j=1}^s X_j\lra M_i$, where $1\leq s\leq n$. So, there is non-negative integer $b'$,  such that $l_R(M_i/\xx^2 M_i)\leq l_R(\oplus_{j=1}^nX_j/\xx^2\oplus_{j=1}^nX_j)\leq b'$. Set $r=2^b$. By Corollary 14.10 and Corollary 14.12, we have the following chain
%\[
% M_1/\xx^2M_1\st{\overline{f_1}}\lra M_2/\xx^2M_2\st{\overline{f_2}}\lra\cdots\lra M_{r-1}/\xx^2M_{r-1} \st{\overline{f}_{r-1}}\lra M_r/\xx^2 M_r
%\]
%of indecomposable objects $M_i/\xx^2M_i$ and non-isomorphisms
%$\overline{f_i}$. According to \cite[Th. 14.20]{lw1}, $f_{r-1}\cdots f_1\otimes_{\Lambda}\Lambda/\xx^2 =0$.}}

By Lemma \ref{L:bound-sop} and~\cite[Corollary 15.11]{leuschke2012cohen}, there is an $\mathcal{L}$-faithful s.o.p. $\xx$ such that, for any indecomposable object $M_i$ in $\mathcal{L}$, the module $M_i/\xx^2 M_i$ is indecomposable.  By Lemma~\ref{L:betti}, $\beta(\mathcal{L}) \leq b$ for some integer $b>0$. By the Harada-Sai lemma based on betti numbers, we are done.
\end{proof}

Notice that the number $r$ in Lemmas \ref{L:Harada-Sai} and \ref{T:Harada-Sai}, is independent of $\mathcal{L}$. In fact, it depends on the faithful s.o.p. $\xx$ and $|\underline{\h}(\mathcal{L})|$.
%\begin{remark}
% The proof shows that, once the bound $p$ is fixed, $\xx$ can be chosen independent of $\T$.
%\end{remark}

\section{Orders of bounded lattice type}
%$\underline{h}$-length

In this section we shall use the Harada-Sai lemma to prove the first Brauer-Thrall theorem for lattices.
\medskip

Recall that the transpose $\Tr (M)$ of a finitely presented $\Lambda$-module $M$ is defined as 
\[
\cok(\Hom_{\Lambda}(P_0,\Lambda)\lra\Hom_{\Lambda}(P_1,\Lambda))
\]
where $P_1\lra P_0\lra M\lra 0$ is a minimal projective presentation of $M$. The next two results, dealing with the existence of almost split sequences in the category of lattices, can be proved the same way as \cite[Chapter II, Propositions 8.2 and 8.5]{auslander1978functors} (by replacing $R$ with 
$\omega$)

\begin{lemma}\label{L:a.s.s1}
\begin{enumerate}
 \item If $M$ is an indecomposable non-projective module in $\Lambda$-$\textbf{lat}$, then there is an almost split sequence 
\[
0\lra \tau M \lra L\lra M\lra 0 
\]
in $\Lambda$-$\textbf{lat}$, where $\tau M \simeq \Hom_R(\Omega^d\,\Tr\,M, \omega)$.

\item If $N$ is an indecomposable $\Lambda$-lattice which is not $\textbf{lat}$-injective, then there is an almost split sequence 
\[
0\lra N\lra L\lra \tau^{-} N \lra 0 
\]
in $\Lambda$-$\textbf{lat}$, where 
$\tau^{-} N \simeq \Omega^d\,\Tr\,\Hom_R(N,\omega)$.
\end{enumerate} 
\qed
%\item\textcolor{red}{If $0\lra N\lra L\lra M\lra 0$ is an almost split sequence in $\Lambda$-$\textbf{lat}$.}
%\end{enumerate}
%\item\textcolor{red}{If $0\lra N\lra L\lra M\lra 0$ is an almost split sequence in $\Lambda$-$\textbf{lat}$.}
%\end{enumerate}
\end{lemma}

The following result can be obtained by the same argument as in
 \cite[Proposition 13.16]{leuschke2012cohen} (given there for $\mcm$ $R$-modules).

\begin{lemma}\label{lem2}
Let 
\[
0\lra K \st{i}\lra L \st{p}\lra M \lra 0
\]
be an almost split sequence in $\Lambda$-$\textbf{lat}$ and 
$N \in \Lambda$-$\textbf{lat}$.
\begin{enumerate}
 \item A $\Lambda$-homomorphism $f:K \lra N$ is irreducible if and only if $f$ is a direct summand of~$i$.
 \item A $\Lambda$-homomorphism $g: N \lra M$ is irreducible if and only if $g$ is a direct summand of~$p$.
\end{enumerate} \qed
\end{lemma}

\begin{lemma}\label{proj-case}
Let $M$ be an indecomposable lattice and $P$ an indecomposable finitely generated projective $\Lambda$-module. If
$f : M \longrightarrow P$ is irreducible in $\Lambda$-$\textbf{lat}$, then $M$ is a direct summand of a minimal $\textbf{lat}$-approximation of $\ra\, P$.
\end{lemma}

\begin{proof}\label{irred}
 First, notice that $f(M) \subset \ra P$, for otherwise
 %$M$ would have a nonzero projective direct summand and, being indecomposable, would coincide with that direct summand, making
 $f$ would be a retraction, which is impossible because $f$ is irreducible. Since $\ra P$ is the first syzygy module of the simple $\Lambda$-module
 $P/\ra P$, by Proposition \ref{P:approximation}, there is a minimal $\textbf{lat}$-approximation $q : X \longrightarrow \ra P$. We now have a commutative diagram of solid arrows;

\[
\xymatrix
	{
	M \ar[r]^{f} \ar[rd]^{f'} \ar@{-->}[d]_{h}
	& P
\\
	X \ar[r]^{q}
	& \ra\, P \ar@{^{(}->}[u]_{\iota}	
	}
\]
Since $M$ is a lattice, there is $h : M \to X$ making the lower triangle commute and giving rise to a factorization of $f$. Since $\iota$ is not epic, the same is true for $\iota q$, hence $h$ must be a section.
\end{proof}

Notice that $f$ in the above lemma is a direct summand of $\iota q$.
\begin{cor}\label{C:irred}
Let $M$ be an indecomposable object in $\add\,\Hom_R(\Lambda_{\Lambda},\omega)$ (i.e., $M$ is an indecomposable 
$\textbf{lat}$-injective). Then there are only finitely many (up to isomorphism)
indecomposable lattices $X$ admitting an irreducible morphism $f:M\lra X$.
\end{cor}

\begin{proof}
If $f: M\lra X$ is irreducible, then so is $f^*:X^*\lra M^*$. By Lemma \ref{proj-case}, $X^*$ is a direct summand of minimal $\textbf{lat}$-approximation of $\ra M^*$, which yields the desired result.
\end{proof}

 %The following generalization of Lemma~\ref{proj-case} is immediate.

%\begin{lemma}
%Let $M$ and $N$ be in $\Lambda$-$\textbf{lat}$. Assume that $M$ is indecomposable and let $f : M \longrightarrow N$ be an irreducible map in
%$\Lambda$-$\textbf{lat}$. Suppose $f$ admits a factorization $M \overset{h}\longrightarrow D \overset{\kappa}\longrightarrow N$ in $\md\Lambda$, where $D$ has an $\textbf{lat}$-approximation $q : X_{D} \longrightarrow D$ and $\kappa$ is not a split epimorphism.  Then a lifting of $h$ over $q$ is a split monomorphism $M \longrightarrow X_{D}$. \qed?????????????????????????
%\end{lemma}

%\begin{lemma}\label{lem6}
%Let $P$ be an indecomposable projective object of ${\CM}$.
%\fbox{What does this mean?}
%A $\Lambda$-homomorphism $f:M \lra P$ is irreducible
%%$\Lambda$-homomorphism
%if and only if $f$ is a direct summand of $q:X \lra L$
%where $L$ is the maximal submodule of $P$ and $X$ is ${\CM}$-cover of $L$.
%\fbox{How is this different from Lemma 3.2?}
%\end{lemma}

%{\bf Note :} In the case $M= Hom_R(\Lambda,w)$ or $N$ is projective ?????
\begin{remark}\label{R:fin-many}
Lemmas \ref{lem2}, \ref{proj-case}, and Corollary \ref{C:irred} show that there are only finitely many (up to isomorphism) irreducible morphisms in $\Lambda$-$\textbf{lat}$ starting or ending at an indecomposable lattice.
%
%objects $N, X \in \Lambda$-$\textbf{lat}$, there are only finitely many indecomposables objects $M, Y \in \Lambda$-$\textbf{lat}$ admitting irreducible maps $f : M \lra N$ and $g:X\lra Y$.
\end{remark}

 \begin{lemma}\label{lem1}
Let $M$ be an indecomposable lattice and $\mathcal{L}$ the class of (representatives of isoclasses of) all indecomposable lattices that $M$ is connected to.
If $\mathcal{L}$ is of bounded $\underline{\h}$-length and $\xx$ is a faithful s.o.p. for $\mathcal{L}$, then the subclass $\mathcal{L}'$ of $\mathcal{L}$
of objects $M$ is connected to modulo $\xx^2$ is finite.
\end{lemma}

 \begin{proof} Since $\mathcal{L}$ is of bounded $\underline{\h}$-length, there is a faithful s.o.p. $\xx$ for $\mathcal{L}$. By Remark~\ref{R:fin-many}, there are only finitely many indecomposable lattices $L_i$ admitting irreducible maps $f_i:M \to L_i$. For the same reason, for any $i$, there are only finitely many lattices $N_{j}'$ admitting irreducible maps $g_j:L_i \to N_{j}'$. $M$ may not be connected to all of the $N_{j}$ modulo $\xx^2$ (some compositions starting at $M$ may conceivably be zero modulo $\xx^2$), but those which $M$ is still connected to constitute a finite set. We can try and repeat this process but, by the Harada-Sai lemma \ref{T:Harada-Sai}, it will terminate after finitely many steps.
\end{proof}

 %\begin{lemma}\label{lem5}
%Let $M$ be an  object of   $\mathcal{C M }$  then  ${\Hom}_{\Lambda}(M, ({\Hom}_R(\Lambda,w))\neq0$.
%\end{lemma}
%\begin{proof} note that ${\Hom}_R(M, w)\neq0$
%\end{proof}

\begin{prop}\label{prop1}
Let $(\xx)\subseteq\m$ be an ideal of $R$ and let $M$ be an 
%\marginpar{\fbox{nonzero???}}
indecomposable lattice such that no indecomposable projective $\Lambda$-module is connected to it modulo $\xx$. 
%\marginpar{\fbox{modulo $\xx$?}}
Then there is an infinite chain
\[
\cdots\lra M_2\st{h_2}\lra M_1\st{h_1}\lra M
\]
of indecomposable lattices $M_i$ and irreducible maps $h_i$ such that
\[
h_{1}\cdots h_{n}\otimes_{\Lambda}\Lambda/\xx\Lambda \neq 0
\]
for any $n\in\mathbb{N}$.
\end{prop}

\begin{proof}
We induct on $n$. For the induction base, notice that, by assumption, $M$ is not projective. By Lemma \ref{L:a.s.s1}, there exists an almost split sequence  
\[
0\lra \tau M\lra E\st{h}\lra M\lra 0
\] 
in $\Lambda$-$\textbf{lat}$. The middle term decomposes as 
$E = \oplus M_{i,1}$, where each summand is indecomposable and each restriction $h_{i,1}$ of $h$ to $M_{i,1}$ is irreducible. Pick any epimorphism from a finitely generated projective onto $M$. By Nakayama's lemma (over $R$), there exist an indecomposable finitely generated projective 
$\Lambda$-module~$P$ and $f : P \to M$ such that $f \otimes \Lambda/\xx\Lambda \neq 0$. Since~$M$ is not projective, $f$ cannot be a retraction and hence lifts over $h$. Thus $f = hg_{1}$ for some 
$g_{1} : P \to E$. It follows that, for some $i$, $h_{i,1}g_{1} \otimes \Lambda/\xx\Lambda \neq 0$. Now we set $M_{1} : = M_{i,1}$ and  $h_{1} := h_{i,1}$. This completes the induction base. 

For the induction step, assume that we have constructed a chain 
\[
M_{n-1} \overset{h_{n-1}}\lra \ldots \overset{h_{2}}\lra M_{1} \overset{h_{1}}\lra M
\]
of irreducible maps and indecomposable lattices, together with $g_{n-1} : P \to M_{n-1}$ such that
\[
h_{1}\ldots h_{n-1}g_{n-1}\otimes \Lambda/\xx\Lambda \neq 0 
\]
By assumption, $M_{n-1}$ is not projective, and therefore  there is an almost split sequence in $\Lambda$-$\textbf{lat}$ ending with $M_{n-1}$. Similar to the above, we construct $h_{n} : M_{n} \to M_{n-1}$ and 
$g_{n} : P \to M_{n}$, where $M_{n}$ is indecomposable, $h_{n}$ is irreducible,
and $h_{1}\ldots h_{n}g_{n}\otimes \Lambda/\xx\Lambda \neq 0$. It follows
that $h_{1}\ldots h_{n} \otimes \Lambda/\xx\Lambda \neq 0$, and we are done by induction.
\end{proof}

\begin{definition}\label{def1}
We say that the order $\Lambda$ is of \texttt{bounded lattice type} if the category $\ind(\Lambda$-$\textbf{lat})$ is of bounded $\underline{\h}$-length.
\end{definition}
Now, we state a version of the first Brauer-Thrall theorem for the category of lattices.

\begin{theorem}\label{bounded}
If $\Lambda$ is of bounded lattice type,
then $\Lambda$ is of finite $\textbf{lat}$-type.
\end{theorem}

\begin{proof}
Since $\ind(\Lambda$-$\textbf{lat})$ is of bounded $\underline{\h}$-length, there is a faithful s.o.p. $\xx$ for $\Lambda$-$\textbf{lat}$. For each indecomposable projective $\Lambda$-module we consider the indecomposable lattices to which the projective connects modulo $\xx^2$. Let~$\mathcal{L}$ denote the union of those classes. By the assumption, $\mathcal{L}$ is of bounded $\underline{\h}$-length. Lemma~\ref{lem1} and the fact that there are only finitely many finitely generated indecomposable projectives show that $\mathcal{L}$ is of finite type. Suppose now that $\Lambda$ is of infinite $\textbf{lat}$-type and pick a lattice $M$ from the complement of $\mathcal{L}$. By Proposition~\ref{prop1}, there is an infinite chain 
\[
\cdots\lra M_2\st{h_2}\lra M_1\st{h_1}\lra M_0\st{h_0}\lra M
\]
of irreducible maps $h_i$ in the complement of $\mathcal{L}$ such that $h_0\cdots h_n\otimes_{\Lambda}\Lambda/\xx^2\Lambda\neq 0$ for any~$n$.
%For any $M_i\in\T$, there is a nonzero $f_{i} \in \Hom_{\Lambda}(M_i,\Hom_R(\Lambda^{\op},\omega))$ such that $f_i$ is nonzero modulo $\xx^2$, where $\xx$ is a $\T$-faithful system of parameters.
%
%By Remark~\ref{R:fin-many}, there are only finitely many $M_i$ connected to indecomposable
%direct summands of $\Hom_R(\Lambda^{\op},\omega)$, and therefore there is an infinite subset $\T'\subseteq\T$ consisting of pairwise non-isomorphic indecomposable objects $M_i$ of $\T$ which are not connected to an indecomposable direct summand of $\Hom_R(\Lambda^{\op},\omega)$. By Proposition~\ref{prop1},
On the other hand, since~$\Lambda$-$\textbf{lat}$ is of bounded $\underline{\h}$-length, Lemma \ref{T:Harada-Sai} yields an integer $r>0$ such that $h_0\cdots h_r\otimes_{\Lambda}\Lambda/\xx^2\Lambda=0$, a contradiction.
%\textcolor{green}{The proof is finished.}
\end{proof}

Recall that $R$ is said to be of finite mCM type if there are only finitely many classes of non-isomorphic indecomposable $\mcm$ $R$-modules. $R$ is said to be of bounded mCM type, if there is a bound on the multiplicities of the indecomposable $\mcm$ $R$-modules. The following result can be found in~\cite[Corollary 5.3]{bahlekehpreprint}.

\begin{cor}\label{cor1}
Let $(R,\m)$ be a complete Cohen-Macaulay local ring. If $\ind(\mcm)$ is of bounded $\underline{\h}$-length (we set $\Lambda := R$), then $R$ is of finite mCM type.
\end{cor}

\begin{proof}
By the assumption $l_R(\underline{\h}(M))$ is finite, for any $\mcm$ $R$-module $M$. So $R$ is an isolated singularity (i.e., $R_{\p}$ is a regular local ring for all non-maximal prime ideals $\p$ of $R$). Hence the category $\mcm$ coincides with the category of lattices over $R$. Theorem~\ref{bounded} now completes the proof.
\end{proof}
The above corollary allows  us to recover results of 
Dieterich~\cite{dieterich1981representation}, Leuschke and Wiegand~\cite{leuschke2013brauer}, and
Yoshino~\cite{yoshino1987brauer}. Before that, we state two lemmas that are needed to prove that result. 

\begin{lemma}\label{lem11}
Let $\mathcal{L}$ be a subcategory of $\mathsf{ind}(\Lambda$-$\textbf{lat})$.
The following conditions are equivalent:
\begin{enumerate}
 \item\label{L} $\beta(\mathcal{L})$ is bounded;
 
 \item\label{R} $\beta_R(\mathcal{L})$ is bounded;
 
 \item\label{M} $e(\mathcal{L})$ is bounded.
\end{enumerate}

If the above equivalent conditions hold, then:

\begin{enumerate}[resume]

\item\label{O} $\beta(\Omega\mathcal{L})$ is bounded, where 
$\Omega\mathcal{L}=\{\Omega M\, |\, M\in\mathcal{L}\}$;

\item\label{T} $\beta(\tau\mathcal{L})$ is bounded, where 
$\tau\mathcal{L}=\{\tau M\, |\, M\in\mathcal{L}\}$.
 
\end{enumerate}
\end{lemma}

\begin{proof}
Let $s$ be the largest among the $R$-betti numbers of the principal 
$\Lambda$-projectives. If $M$ is a finitely generated $\Lambda$-module then 
$\beta_{R}(M) \leq s \beta(M)$. Combining this with the obvious inequality
$\beta(M) \leq \beta_{R}(M)$, we have the equivalence of \eqref{L} and \eqref{R}.

To prove the implication \eqref{R} $\Rightarrow$ \eqref{M}, assume that there is an integer $s>0$ such that $\beta_R(M) \leq s$ for all $M\in\mathcal{L}$.
Then, for any $M\in\mathcal{M}$, there is an epimorphism $R^s \to M$, and  $e(M) \leq se(R)$. 
%This implies that there is an integer $b>0$ such that $l_R(M/\xx M) \leq b$. \fbox{$\xx$ is undefined?????}. Therefore, \cite[Proposition 1.7]{yoshino1990cohen} yields that $e(M)\leq b$, for any $M\in\mathcal{L}$. 
The reverse implication follows from the fact that 
$\beta_R(M)\leq e(M)$ for any $M\in\mathcal{L}$, see \cite[Corollary A.24]{leuschke2012cohen}

\eqref{L} $\Rightarrow$ \eqref{O}. Assume that $\beta(\mathcal{L})$ is bounded, i.e., the number of indecomposable direct summands appearing in the projective cover of any object in $\mathcal{L}$ is less than some fixed number. Hence the multiplicities of the projective covers of objects in $\mathcal{L}$ are bounded and so are the multiplicities of the corresponding syzygy modules. 
%That is $e(\Omega\mathcal{L})$ is bounded. 
The implication \eqref{M} $\Rightarrow$ \eqref{L} shows that 
$\beta(\Omega\mathcal{L})$ is bounded.

\eqref{L} $\Rightarrow$ \eqref{T}. First, we show that there is $m>0$ such that for all $M\in\mathcal{L}$, $\beta_{\Lambda^{\op}}(\Tr M)<m$. Applying the functor $(-)^{\ast}=\Hom_{\Lambda}(-,\Lambda)$ to a minimal presentation 
$Q \to P \to M \to 0$, we have a presentation 
\[
 P^{\ast}\lra Q^{\ast} \lra \Tr M \lra 0.
\]
By the implication \eqref{L} $\Rightarrow$ \eqref{O}, we have that 
$\beta(\Omega\mathcal{L})$ is bounded.
Consequently, the number of indecomposable summands of $Q^{\ast}$ is bounded, i.e., $\beta_{\Lambda^{\op}}(\Tr M)<m$ for some $m>0$ independent of $M$. By the implication \eqref{L} $\Rightarrow$ \eqref{O}, there is an integer $s$ such that $\beta_{\Lambda^{\op}}(\Omega^d\Tr M)\leq s$ for all 
$M \in \mathcal{L}$. Hence, for any $M\in\mathcal{L}$, there exists an epimorphism $\Lambda_{\Lambda}^s\lra \Omega^d\Tr M$. This gives rise to a monomorphism 
\[
\Hom_R(\Omega^d\Tr M,\omega)\lra \Hom_R(\Lambda^s,\omega).
\]
Thus there is integer $b>0$ such that $e(\Hom_R(\Omega^d\Tr M,\omega))\leq b$ for any $M\in\mathcal{L}$, i.e., $e(\tau\mathcal{L})$ is bounded.
%(1) It is obvious from definition.
%(2) By \cite[Proposition 4.1]{krause2013morphism} $\Lambda=P_1\oplus\cdots\oplus P_t$. So set $\{P_1, \cdots, P_t\}$ is set of all
 %indecomposable projective $\Lambda$-modules. By assumption there exists an integer $n>0$ such that $\beta(\mathcal{L})<n$.
 %Thus for any $M_j \in\mathcal{L}$ there is an epimorphism $f:\oplus_{i=1}^tP_i^{n_i}\lra M_j$ such that $\Sigma_{i=1}^tn_i<n$. Assume that $\xx$ is an $M_j$-regular sequence. Tensoring $f$ with
 %$\Lambda/\xx\Lambda$ over $\Lambda$, gives rise to an epimorphism $\overline{f}:\oplus_{i=1}^tP_i^{n_i}/\xx\oplus_{i=1}^tP_i^{n_i}\lra  M_j /\xx M_j $. Thus $l_R(M_j /\xx M_j)$ is bounded, and then $e(\mathcal{L})$ is bounded by~\cite[Proposition 1.7]{yo}.
 %====
 %For the converse, assume that $e(\mathcal{L})$ is bounded. Then $l_R(M_i/\xx M_i)$, where  $\xx$ is a system of parameters, is bounded by
 %\cite[Proposition 1.7]{yo}. So $\mu_{R/\xx R}(M_i/\xx M_i)$, for any $M_i\in\mathcal{L}$, is bounded and since $\mu_{R/\xx R}(M_i/\xx M_i)=  \mu_R(M_i)$, we have that $\mu_R(M_i)$, for any $M_i\in\mathcal{L}$, is bounded. According to (1), $\beta(\mathcal{L})$ is bounded.
\end{proof}
\begin{lemma}\label{L:lem12}
Let $\mathcal{L}$ be a subcategory of $\Lambda$-$\textbf{lat}$ and $\xx=x_1, \ldots, x_d$ be a faithful s.o.p.
for $\mathcal{L}$. If
$\beta(\mathcal{L})$ is bounded, then $\mathcal{L}$ is of bounded $\underline{\h}$-length.
\end{lemma}

\begin{proof}
For notational efficiency, let overline denote reduction modulo $\xx\Lambda$.
We first claim that 
%$l_R(\Hom_{\Lambda/\xx\Lambda}(M/\xx M,M\oplus G/\xx (M\oplus G)))$
$l_R(\Hom_{\overline{\Lambda}}(\overline{M}, \overline{M\oplus G}))$ 
is bounded when $M$ runs through $\mathcal{L}$. By assumption, there are non-negative integers $n_i$ such that, for any $M\in\mathcal{L}$, there are 
$\Lambda$-epimorphisms 
\[
f : \oplus_{i=1}^tP_i^{n_i} \lra M \quad \textrm{and} \quad 
g : \oplus_{i=1}^tP_i^{n_i} \lra M\oplus G
\]
 where the $P_i$ are principal projective $\Lambda$-modules. Tensoring~$f$ and $g$ with $\Lambda/\xx\Lambda$ over $\Lambda$ gives rise to epimorphisms
 \[
 \overline{f} :\oplus_{i=1}^t \overline{P}_i^{n_i} \lra \overline{M} \quad \textrm{and}
 \quad \overline{g} : \oplus_{i=1}^t \overline{P}_i^{n_i} \lra \overline{M \oplus G}.
 \]
%$\overline{f}:\oplus_{i=1}^tP_i^{n_i}/\xx(\oplus_{i=1}^tP_i^{n_i})\lra M/\xx M$ 
%and
%$\overline{g}: \oplus_{i=1}^tP_i^{n_i}/\xx(\oplus_{i=1}^tP_i^{n_i})\lra M\oplus G/\xx (M\oplus G)$. 
Applying $\Hom_{\overline{\Lambda}}(-, \overline{M \oplus G})$ to $\overline{f}$, we have a monomorphism 
\[
\Hom_{\overline{\Lambda}}(\overline{M}, \overline{M\oplus G}) \lra 
\Hom_{\overline{\Lambda}}(\oplus_{i=1}^t \overline{P}_i^{n_i}, \overline{M\oplus G}).
%\Hom_{\Lambda/\xx\Lambda}(\oplus_{i=1}^tP_i^{n_i}/\xx(\oplus_{i=1}^tP_i^{n_i}),M\oplus G/\xx(M\oplus G))
\]
Applying $\Hom_{\overline{\Lambda}}(\oplus_{i=1}^t \overline{P}_i^{n_i}, -)$ to
$\overline{g}$, we have, since each $\overline{P}_{i}$ is $\overline{\Lambda}$-projective, an epimorphism
%\[
%\Hom_{\Lambda/\xx\Lambda}(\oplus_{i=1}^t(P_i^{n_i}/\xx P_i^{n_i}),\oplus_{i=1}^t(P_i^{n_i}/\xx P_i^{n_i})) \lra 
%\Hom_{\Lambda/\xx\Lambda}(\oplus_{i=1}^t(P_i^{n_i}/\xx P_i^{n_i}),M\oplus G/\xx(M\oplus G)),
%\]
\[
\Hom_{\overline{\Lambda}}(\oplus_{i=1}^t \overline{P}_i^{n_i}, 
\oplus_{i=1}^t \overline{P}_i^{n_i}) \lra 
\Hom_{\overline{\Lambda}}(\oplus_{i=1}^t \overline{P}_i^{n_i}, \overline{M\oplus G}).
\]
It now follows that
%\[
%l_R(\Hom_{\Lambda/\xx\Lambda}(M/\xx M,M\oplus G/\xx (M\oplus G)))\leq l_R(\Hom_{\Lambda/\xx\Lambda}(\oplus_{i=1}^t(P_i^{n_i}/\xx P_i^{n_i}),(\oplus_{i=1}^t(P_i^{n_i}/\xx P_i^{n_i})).
%\]
\begin{equation}\label{b}
 l_R(\Hom_{\overline{\Lambda}}(\overline{M}, \overline{M\oplus G})) \leq 
l_R(\Hom_{\overline{\Lambda}}(\oplus_{i=1}^t \overline{P}_i^{n_i}, 
\oplus_{i=1}^t \overline{P}_i^{n_i}) =: b,
\end{equation}
where $b$ is obviously independent of $M$.
%Setting $b:=l_R(\Hom_{\Lambda/\xx\Lambda}(\oplus_{i=1}^t
%P_i^{n_i}/\xx(\oplus_{i=1}^tP_i^{n_i}),\oplus_{i=1}^t P_i^{n_i}/\xx(\oplus_{i=1}^tP_i^{n_i})))
%$. \textcolor{green}{We obtain that,
%for any $M\in\mathcal{L}$}, we have that $l_R(\Hom_{\Lambda/\xx\Lambda}(M/\xx M,M\oplus G/\xx (M\oplus G)))\leq b$.
On the other hand, for any $0\leq i\leq d-1$, we have an exact sequence 
\[
0\lra M\oplus G/\xx_i(M\oplus G)\st{x_{i+1}}\lra M\oplus G/\xx_i(M\oplus G)\lra M\oplus G/\xx_{i+1}(M\oplus G)\lra 0,
\]
where $\xx_i=x_1,\cdots , x_i$. By the half-exactness of Hom modulo projectives, this induces an exact sequence
\[
\underline{\Hom}_{\Lambda}(M,M\oplus G/\xx_i(M\oplus G))\st{x_{i+1}}\lra\underline{\Hom}_{\Lambda}(M,M\oplus G/\xx_i(M\oplus G))\st{\varphi}\lra
\underline{\Hom}_{\Lambda}(M,M\oplus G/\xx_{i+1}(M\oplus G)).
\]
%see \cite[4.17.2]{yoshino1990cohen}.
Since $\xx=\xx_d$ is a faithful s.o.p. for $\mathcal{L}$, 
\[
x_{i+1}(\underline{\Hom}_{\Lambda}(M,M\oplus G/\xx_i(M\oplus G)))=0
\]
and hence $\varphi$ is monic. Now, inducting on $i$, we have
\[
l_R(\underline{\Hom}_{\Lambda}(M, M\oplus G)) \leq 
l_R(\underline{\Hom}_{\Lambda}(M, \overline{M\oplus G})).
\]
This, coupled with the obvious inequality
\[
l_R(\underline{\Hom}_{\Lambda}(M, \overline{M\oplus G})) \leq l_R(\Hom_{\Lambda}(M, \overline{M\oplus G}))
\]
yields 
\[
l_R(\underline{\Hom}_{\Lambda}(M,M\oplus G)) \leq 
l_R(\Hom_{\Lambda}(M, \overline{M\oplus G})).
\]
%By \cite[Lemma 2(i), p. 140]{matsumura1989commutative}, we have an isomorphism
Since 
\[
\Hom_{\Lambda}(M, \overline{M\oplus G}) \simeq 
\Hom_{\overline{\Lambda}}(\overline{M}, \overline{M\oplus G}),
\]
we have, in view of~\eqref{b}, that 
$l_R(\underline{\Hom}_{\Lambda}(M,M\oplus G)) \leq b$.
\end{proof}

Now we can prove the promised result.

\begin{cor}\label{cor5}
Let $(R,\m)$ be a complete equicharacteristic Cohen-Macaulay local ring with algebraically closed residue field $k$. Then $R$ is of finite mCM type if and only if $R$ is of bounded mCM type (i.e., the betti numbers or, equivalently, the multiplicities of the pairwise nonisomorphic indecomposable mCM modules are bounded) and is an isolated singularity or regular.
\end{cor}

\begin{proof}
If $R$ is of finite mCM type, then it is clearly of bounded mCM type. Moreover, as was shown by Auslander~\cite{auslander1986isolated}, $R$ is an isolated singularity or regular.
Now we prove the `if' part. By Corollary \ref{cor1}, it suffices to show that the category of all indecomposable $\mcm$ $R$-modules is of bounded 
$\underline{\h}$-length. By \cite[Theorem 15.18]{leuschke2012cohen}, $R$ admits a faithful system of parameters $\xx$. Moreover, by the hypothesis, there
is an integer $b>0$ such that $e(M)<b$ for any indecomposable $\mcm$ $R$-module $M$. Applying now Lemmas~\ref{lem11} and~\ref{L:lem12}, we have the desired result.
\end{proof}

%On the other hand, $\xx$ is an $M\oplus G$-sequence and the exact sequence $0\lra M\oplus G\st{\xx}\lra M\oplus G\lra M\oplus G/\xx M\oplus G
%\lra 0$ induces
%\[
%\underline{\Hom}_{\Lambda}(M,M\oplus G)\st{\xx}\lra \underline{\Hom}_{\Lambda}(M,M\oplus G)\st{\phi}\lra\underline{\Hom}_{\Lambda}(M,M\oplus G/\xx M\oplus G)/
%\]
%Since $\xx$ is a $\T$-faithful s.o.p., $\phi$ is monic.

The following theorem shows that the category $R$-$\textbf{lat}$ contains indecomposable lattices of arbitrarily large $\underline{\h}$-length whenever $R$ is an abstract hypersurface. The latter is defined as a Noetherian local ring $(R,\m)$ such that its $\m$-adic completion $\hat{R}$ is isomorphic to $S/(f)$ for some regular local ring $S$ and $f \in \m^{2}$.

\begin{theorem}\label{kavazaki}
Let $(R,\m)$ be an abstract hypersurface of dimension $d\geq 2$. If $e(R) > 2$, then there are indecomposable lattices of arbitrarily large (finite) $\underline{\h}$-length.
\end{theorem}

\begin{proof}
According to \cite[Theorem 4.1]{kawasaki1996local},
for each $n> e(R)$, the maximal Cohen-Macaulay module $\Omega_R^{d+1}(R/\m^n)$ is indecomposable and
\[
\beta(\Omega_R^{d+1}(R/\m^n)) \geq {d+n-1 \choose d-1}
%
%\tiny {\left(\begin{array}{ll} d+n-1 \\
%d-1 \end{array} \right)}
\]
Since $R/\m^n$ is an artinian $R$-module, for any $n$, 
$\Omega_R^{d+1}(R/\m^n)$ is a lattice. Lemma~\ref{L:betti} now shows that the 
$\underline{\h}$-lengths of indecomposable lattices $\Omega_R^{d+1}(R/\m^n)$ are not bounded.
\end{proof}

%\begin{remark}
%Let $(R,m,k)$ be a complete Cohen-Macaulay local ring containing its residue field $k$.
%Assume that $k$ is perfect and that $R$ has an isolated singularity. Then $R$ admits a faithful
%system of parameter; see \cite{y0}.
%\end{remark}
%\begin{cor} ( Theorem of Dietrich  and Yoshino )
%Let $R$ be as above Remark . Assume that $\Gamma ^0$ is a connected componenet of $\Gamma$ and that
%$\Gamma ^0$ is of bounded mCM type. Then $\Gamma=\Gamma^0$ and $\Gamma$ is a finite graph. In
%particular, $R$ is finite mCM type.
%\end{cor}
%\begin{proof}
%\end{proof}

\section{Orders of strongly unbounded lattice type}

In this section, we investigate a Brauer-Thrall $1\frac{1}{2}$ in the category of lattices. We shall show that if there are infinitely many indecomposable lattices of the same $\underline{\h}$-length, then $\Lambda$ has strongly unbounded lattice type. At the end of this section, we present some examples of such $R$-orders. %of strongly unbounded lattice type.

\begin{definition}\label{d:strongly}
We say that an $R$-order $\Lambda$ has \texttt {strongly unbounded} 
$\textbf{lat}$-\texttt{type} if there is an infinite sequence $b_1<b_2<\cdots$ of  integers
%and infinite families
%$\mathcal{L}_{1}, \mathcal{L}_{2}, \mathcal{L}_3, \cdots$ of non-isomorphism indecomposable lattices
such that, for each $i$, there are infinitely many non-isomorphic indecomposable  $\Lambda$-lattices with $\underline{\h}$-length $b_{i}$.
%$|\underline{\h}(\mathcal{L}_{i-1})|<b_i$   and there exist   infinitely many indecomposable lattices $M \in \mathcal{L}_{i} - \mathcal{L}_{{i-1}}$ with $l_R(\underline{h}(M))  = b_{i}$.
\end{definition}

The following lemma should be compared with a result of 
Auslander~\cite[I, Proposition 7.5]{auslander1978functors}.
\begin{lemma}\label{lem4}
Suppose $\dim R = d \geq 1$. If $M$ is a lattice, then $\Ext_{\Lambda^{\op}}^i(\Tr M,\Lambda^{\op}) = 0$ for all $1 \leq i \leq d$.
\end{lemma}

\begin{proof}
Since $M$ is finitely presented, $\Ext_{\Lambda^{\op}}^1(\Tr M,\Lambda)$ is the 1-torsion submodule of $M$, i.e., it is the kernel of the canonical evaluation map $M \lra M^{\ast\ast}$, where $()^{\ast}$ denotes, depending on the position, either 
$\Hom_{\Lambda}(-, \Lambda)$ or $\Hom_{\Lambda^{op}}(-,\Lambda)$. As $M$ is a lattice, $\Ext_{\Lambda^{\op}}^1(\Tr M,\Lambda)$ is of finite length. Since $d \geq 1$, the depth of $M$ is at least 1. It follows that $\Ext_{\Lambda^{\op}}^1(\Tr M,\Lambda)$, being a submodule of $M$, is zero. In particular, the statement of the lemma is true when $d = 1$. Thus we may assume that $d \geq 2$. Since the cokernel of the canonical evaluation map is isomorphic to $\Ext_{\Lambda^{\op}}^2(\Tr M, \Lambda)$, we have a short exact sequence
\[
0 \lra M \lra M^{\ast\ast} \lra \Ext_{\Lambda^{\op}}^2(\Tr M,\Lambda) \lra 0.
\]
As $M_\p$ is $\Lambda_{\p}$-projective for all non-maximal prime ideals $\p$ of $R$,  the above cokernel is of finite length and therefore has depth 0. Since $M^{\ast\ast}$ is a syzygy $\Lambda$-module, it has depth at least 1. It follows that, if the cokernel is nonzero, the depth of $M$ is 1. On the other hand, $M$ is a maximal Cohen-Macaulay $R$-module and therefore its depth equals $d$, which is at least 2. The obtained contradiction shows that $\Ext_{\Lambda^{\op}}^2(\Tr M,\Lambda) = 0$ and that $M$ is reflexive when $d \geq 2$.

%Since $M$ is locally projective on the punctured spectrum of $R$,
%it is a direct summand of $\Omega_{\Lambda}^2(M/{\xx M})$ for some $M$-regular sequence $\xx =x_1, x_2$ (see~\cite[Lemma 2.2]{hp}). Thus \fbox{Need a reference; Evans \& Griffith}
%$M$ \marginpar{$\bigoplus$}
% is a reflexive
%$\Lambda$-module, implying that
%\[
%\Ext_{\Lambda^{\op}}^1({\Tr} M,\Lambda) =0 =\Ext_{\Lambda^{\op}}^2({\Tr} M,\Lambda).
%\]
Applying the functor $\Hom_{\Lambda^{\op}}(-,\Lambda)$
to a projective resolution
\[
\cdots\lra P_d\lra P_{d-1} \lra \cdots \lra P_1\lra P_0\lra M^{\ast}\lra 0
\]
of the $\Lambda^{\op}$-module $M^{\ast}$, gives rise to the following complex of
$\Lambda$-modules,
\[
0 \lra M \lra \Hom_{\Lambda^{\op}}(P_0,\Lambda) \lra \cdots \lra \Hom_{\Lambda^{\op}}(P_{d-1},\Lambda) \lra \Hom_{\Lambda^{\op}} (P_d,\Lambda)
\]
Since $M^{\ast}$ is free on the punctured spectrum of $R$, the homology of this complex has finite length. Since all the modules in this complex have depth $d$ when viewed as $R$-modules, the acyclicity lemma of Peskine-Szpiro makes this complex exact. Since $M^{\ast}$ is projectively equivalent to the second syzygy module of $\Tr M$, the desired result follows by dimension shift.
%all of whose modules are mCM over $R$. By~\cite[Exercise 12.45]{lw1}, this sequence is exact, meaning that
%$\Ext_{\Lambda^{\op}}^{i>0}(M^*,\Lambda)=0$. Hence, the isomorphism $\Ext_{\Lambda^{\op}}^{i-2}(M^*,\Lambda^{\op})\simeq \Ext_{\Lambda^{\op}}^i({\Tr} M,\Lambda^{\op})$ for any $i\geq 3$ gives the desired result.
\end{proof}

%From now we assume that $\T$ is the set of all indecomposable objects  of $\mathcal{L}$ and  $\bigcap_{X\in T} {\mathsf{ann}} ({ \underline{h}^X})$  is $\m$-primary

%We say that ay  subcategory  $\T_1$  of $\mathcal{L}$  is closed under syzygies if for all $X$ in $\T$, $\Omega^i_{\Lambda}X \in \T$, for all $i$ in $\mathbb{Z}$
 %{\bf note} For an integer $b$  there exists a faithful system of parameters  $\xx= x_1,x_2, \cdots, x_d$ of $R$ such that, for all indecomposable objects $M$ of
%For a subcategory $\mathcal{L}$ of lattices, we consider classical invariants
%$\beta, \mu, e$ and in the following result we show that there are bounds on the one of the mentioned invariants
%and others invariant, simultaneously.

\begin{lemma}\label{lem8}
Let $\mathcal{L}$ be the subcategory of $\ind (\Lambda$-$\textbf{lat})$ of bounded $\underline{\h}$-length, and let 
\[
\mathcal{L}' := \{N_j\}_{j\in J} \cup \tau \mathcal{L} \cup \mathcal{L}
\]
where the $N_j$ are the indecomposable objects that appear as the direct summands of the middle terms of the almost split sequences ending at objects of $\mathcal{L}$. Then $\beta ({\mathcal{L}')}$ is bounded.
\end{lemma}

\begin{proof} 
\marginpar{$\bigotimes$}
By Lemma~\ref{L:a.s.s1}, for each $M\in\mathcal{L}$ there is an almost split sequence 
\[
0 \lra \tau M \lra E_M \lra M \lra 0
\]
%where $\tau M = \Hom_{R}(\Omega^d\,\Tr\, M,\omega)$ and $E_M=\oplus_{i\in I}N_i$. 
By Lemma \ref{L:betti}, $\beta(\mathcal{L})$ is bounded. Lemma~\ref{lem11} shows that $\beta(\tau \mathcal{L})$ is bounded, too.
%By Lemma \ref{lem13}, it suffices to show that the
%$\beta(N_i)$ are bounded. By Theorem~\fbox{Give reference},
%\fbox{A-R?} \marginpar{$\bigotimes$}
%for each $M \in \mathcal{L}_{1}$ there is an almost split sequence
%$0\lra \tau M \lra E \lra M \lra 0$, where $E\simeq\oplus_iN_i$.
Since the projective cover of $E$ is a direct summand of the direct sum of
projective covers of $M$ and of~$\tau M$, $\beta(\{E_M\}_{M\in\mathcal{L}})$ is also bounded. This finishes the proof.
%Hence there is an integer $b>0$ such that $\beta(X)<b$ for any $X$ in $\mathcal{L}'$.
\end{proof}

\begin{lemma}\label{asterisk}
 There is a system of parameters $\yy$ of $R$ such that
 \[
 \yy \Ext^{i}_{\Lambda}(-, \Lambda) = 0 \quad \textrm{and } \quad
 \yy \Ext^{i}_{\Lambda^{op}}(-, \Lambda) = 0
 \]
over $\Lambda$-$\textbf{lat}$ and, respectively, $\Lambda^{op}$-$\textbf{lat}$
for all integers $i>0$.
\end{lemma}

\begin{proof}
Since $\Lambda$ is an $R$-order, 
$\Hom_{R}(_{\Lambda}\Lambda, \omega)_{\p}$ is 
 $\Lambda^{\rm op}_{\p}$-projective for any non-maximal prime ideal~$\p$ 
 of $R$. Thus there is a natural number $t$ such that  
 $\m^t\Ext_{\Lambda^{\op}}^1(\Hom_R(_{\Lambda}\Lambda, \omega), -) = 0$. Choosing an s.o.p. in $\mathfrak{m}^t$ and using the duality functor $\Hom_R(-,\omega)$ between the lattice categories over~$\Lambda$ and~$\Lambda^{op}$, we have that that s.o.p. annihilates $\Ext_{\Lambda}^i(-,\Lambda)$ restricted to
 $\Lambda$-$\textbf{lat}$. By lemma~\ref{L:opposite}, $\Lambda^{op}$ is an $R$-order. A similar argument produces an s.o.p. annihilating 
$\Ext_{\Lambda^{op}}^i(-,\Lambda)$ restricted to $\Lambda^{op}$-$\textbf{lat}$. Choosing $\yy$ as the product of the two systems, we have the desired claim.
\end{proof}

\begin{lemma}\label{transpose}
Let $\mathcal{L}$ be a subcategory of $\ind (\Lambda$-$\textbf{lat})$ and let $\xx$ be a faithful s.o.p. for $\mathcal{L}$. Then there is an $\mathcal{L}'$-faithful s.o.p., where $\mathcal{L}'$ is defined as in Lemma~\ref{lem8}
%$=  \{N_j\}_{j\in J} \cup \tau \mathcal{L} \cup \mathcal{L}$ and $N_{j}^,$s are objects appearing in the middle terms of the almost split sequences ending at objects of $\mathcal{L}$.
\end{lemma}

\begin{proof}
By Proposition \ref{prop2}, $\xx\underline{\Hom}_{\Lambda}(M,M)=0$ for any $M \in \mathcal{L}$. Since $\Tr$ is a duality on the category of finitely presented modules modulo projectives,
$\xx\underline{\Hom}_{\Lambda^{\op}}(\Tr M, \Tr M)=0$.
Lemma~\ref{lem4} shows that any endomorphism of $\Omega^{d} \Tr M$ is induced by an endomorphism of $\Tr M$. It now follows that 
\[
\xx\underline{\Hom}_{\Lambda^{\op}}(\Omega^{d}\Tr M, \Omega^{d}\Tr M)=0
\]
%By Proposition~\ref{prop2}, $\xx\underline{\Hom}_{\Lambda}(M,M)=0$ for any $M \in \mathcal{L}_1$.
%We claim that $\underline{\Hom}_{\Lambda}(\tau M_i,-)\simeq \underline{\Hom}_{\Lambda}(M_i,-)$.
%By Lemma~\ref{transpose} and Proposition~\ref{prop2},
%$\xx \underline{\Hom}_{\Lambda^{op}}(\Omega^{d} \Tr M,
%\Omega^{d} \Tr M) = 0$ for any $M \in \T_{1}$.
 In view of Lemma~\ref{L:CM-inj}, the $\omega$-dual of $\Lambda_{\Lambda}$ is $\textbf{lat}$-injective. Thus $\xx \overline{\Hom}_{\Lambda}(\tau M, \tau M) = 0$, where overline stands for ``modulo $\textbf{lat}$-injectives''. In other words, $x 1_{\tau M}$ factors through a $\textbf{lat}$-injective for any $x$ in the ideal generated by $\xx$. On the other hand, since each $x$ is a central element, 
 \[
 x1_{\Ext_{\Lambda}^{1}(-, \tau M)} = \Ext_{\Lambda}^{1}(-, x1_{\tau M})
 \]
 and, as we just saw, the right-hand side vanishes on $\Lambda$-lattices. It follows that 
 $\xx \Ext_{\Lambda}^{1}(-, \tau M) = 0$ on $\Lambda$-lattices for any 
 $M \in\mathcal{L}$. 
 
By Corollary~\ref{C:enough}, the category $\textbf{lat}$ has enough $\textbf{lat}$-injectives. Let $\Sigma\tau M$ denote the first cosyzygy module of $\tau M$ in a $\textbf{lat}$-injective resolution. Thus $\Sigma\tau M$ is a lattice. By Lemma~\ref{L:CM-inj}, we may choose a cosyzygy sequence of the form 
\begin{equation}\label{es1}
0\lra \tau M\lra \Hom_R(\Lambda_{\Lambda}^m,\omega)\lra \Sigma\tau M\lra 0.
\end{equation}
Applying the functor $\Ext_{\Lambda}^{1}(\Sigma \tau M, -)$
to a syzygy sequence
\[
0 \lra \Omega \tau M \lra \Lambda^{n} \lra \tau M \lra 0
\]
of $\tau M$, we have an exact sequence
\[
\Ext_{\Lambda}^{1}(\Sigma \tau M, \tau M) \lra
\Ext_{\Lambda}^{2}(\Sigma \tau M, \Omega \tau M) \lra
\Ext_{\Lambda}^{2}(\Sigma \tau M, \Lambda^{n}).
\]
By Lemma~\ref{asterisk}, there is a system of parameters $\yy$ annihilating
$\Ext_{\Lambda}^{i}(-, \Lambda^{n})$ and $\Ext_{\Lambda^{\op}}^i(-,\Lambda_{\Lambda})$. In this proof, we already remarked that $\xx$ annihilates $\Ext^{1}(\Sigma \tau M, \tau M)$.
Thus the middle term is annihilated by~$\xx \yy$.

Applying $\Ext_{\Lambda}^{1}(-, \Omega \tau M)$ to the cosyzygy sequence \eqref{es1} gives rise to an exact sequence
\[
\Ext_{\Lambda}^{1}((\Hom_R(\Lambda_{\Lambda}^{m}, \omega), \Omega \tau M) \lra
\Ext_{\Lambda}^{1}(\tau M, \Omega \tau M) \lra
\Ext_{\Lambda}^{2}(\Sigma \tau M, \Omega \tau M).
\]
As we have just mentioned, the last term of this sequence is annihilated by $\xx \yy$. Recalling that $\yy\Ext_{\Lambda^{\op}}^i(-, \Lambda)=0$ and using the duality functor $\Hom_R(-, \omega)$ between the lattice categories over $\Lambda$ and $\Lambda^{\op}$, we have 
\[
\yy\Ext_{\Lambda}^i(\Hom_R(\Lambda_{\Lambda}^m,\omega),\Omega\tau M)=0
\]
As a consequence, the middle term, $\Ext_{\Lambda}^{1}(\tau M, \Omega \tau M)$ is annihilated by $\xx \yy^2$ for any $M \in \mathcal{L}$. In particular, 
$\xx \yy^2$ annihilates the element corresponding to a syzygy sequence for 
$\tau M$. Since $\xx \yy^2$ is in the center of $\Lambda$ and any extension of 
$\tau M$ is a pushout of the syzygy sequence, $\xx\yy^2\Ext_{\Lambda}^1(\tau M,-)=0$ on the entire category of $\Lambda$-modules. 

Thus $\xx\yy^2$ is faithful for both $\mathcal{L}$ and $\tau\mathcal{L}$. Then it is clearly faithful for the middle terms of all almost split sequences ending at 
$\mathcal{L}$, and therefore for the entire $\mathcal{L}'$.
\end{proof}

\begin{cor}\label{cor4}
Let $\mathcal{L} \subseteq\ind (\Lambda$-$\textbf{lat})$ be of bounded $\underline{\h}$-length. Let $\mathcal{L}'$ be the union of $\mathcal{L}$ and the class of domains of all irreducible morphisms with
codomains in $\mathcal{L}$. Then $\mathcal{L}'$ is of bounded $\underline{\h}$-length.
%
%
%an infinite family  of pairwise non-isomorphic objects in $\mathsf{ind}\mathcal{L}$  such that, for all $i$, $|\underline{h}^{M_i}|\leq b$ in $\T_1$. Let for all $i$,  $f_i  :M_i \lra N_i$ be  irreducible morphisms between indecomposable objects in $\mathcal{L}$. Then $| \underline{h}^{\T_{2}}(\T_{2}) |$ is bounded, where  $\T_2=\{N_i\}_{i\in I} \cup \T_1$.
\end{cor}
\begin{proof}
As $\mathcal{L}$ is of bounded $\underline{\h}$-length,
there is an $\mathcal{L}$-faithful s.o.p. $\xx$. The domains of the irreducible morphisms in question are direct summands of the middle terms of the almost split sequences ending at objects of $\mathcal{L}$. By Lemma \ref{transpose}
, there is an $\mathcal{L}'$-faithful s.o.p., where
$\mathcal{L}' := \mathcal{L} \cup \tau\mathcal{L} \cup \{N_{j}\}_{j \in J}$ and the $N_j$ are the objects that appear as indecomposable direct summands of the middle terms of the almost splits sequences ending at objects of $\mathcal{L}$. On the other hand, Lemma \ref{lem8} shows that $\beta({\mathcal{L}'})$ is bounded. By Lemma~\ref{L:lem12},
$\mathcal{L}'$ is of bounded $\underline{\h}$-length.
\end{proof}

\begin{theorem}\label{T:BT-2}
%Let $\mathcal{L}$ be a subclass of 
If $\Lambda$-$\textbf{lat}$ contains  infinitely many non-isomorphic indecomposable objects of the same $\underline{\h}$-length $b>0$, then there exists an integer $t>b$ such that there are infinitely many non-isomorphic indecomposable lattices of $\underline{\h}$-length $t$.
\end{theorem}

\begin{proof}
Let $\mathcal{L}$ be a subcategory of $\Lambda$-$\textbf{lat}$ containing infinitely many non-isomorphic indecomposable lattices of $\underline{\h}$-length $b$ and $\mathcal{M} : = \{M_{i}\}_{i \in I}$ the subclass of $\mathcal{L}$ consisting of those indecomposable lattices in~$\mathcal{L}$ which no projective $\Lambda$-module connects to. Without loss of generality, assume that the $M_{i}$ are pairwise non-isomorphic. Since $\mathcal{L}$ is of infinite type, either $\mathcal{M}$ itself or the complement of~$\mathcal{M}$ in $\mathcal{L}$ is of infinite type. Accordingly, we have two cases.

\textbf{Case 1}. Assume that $\mathcal{M}$ is of infinite type. By Lemma~\ref{L:bound-sop}, there is an s.o.p. $\xx$ which is faithful for all of indecomposable lattices of $\underline{\h}$-length less than $b+1$. By Proposition~\ref{prop1}, for each~$M_{i}$ in $\mathcal{M}$, there is an infinite chain
\[
\cdots\lra M_{i,2}\st{h_{i,2}}\lra M_{i,1}\st{h_{i,1}}\lra M_i
\]
of indecomposable lattices and irreducible maps $h_{i,j}$ such that
$h_{i,1}\cdots h_{i,n}\otimes_{\Lambda}\Lambda/\xx^2 \Lambda \neq 0$
for any integer $n>0$. It is convenient to think of these chains as rows, ending with the $M_{i}$ and indexed by~$I$, of a table $\mathbb{M}$ whose columns are indexed by non-negative integers (the last column has index 0). The following property of $\mathbb{M}$ will be used repeatedly: any isomorphism class contained in a single column of~$\mathbb{M}$ is finite, and in any column there are infinitely many non-isomorphic indecomposable objects. Indeed, this follows from repeated application of Remark~\ref{R:fin-many} and the fact that the last column of $\mathbb{M}$ consists of infinitely many non-isomorphic lattices. As an immediate consequence, we see that any class which is a union of finitely many isoclasses contained in a single column must be finite. We shall refer to this observation as the \textit{finiteness principle}.

Let $\mathcal{M}_{1} := \{M_{i,1}\}_{i \in I}$, i.e., the domains of the irreducible morphisms with codomains in $\mathcal{M}$. By Corollary~\ref{cor4}, 
$\mathcal{M}_{1}$ is of bounded $\underline{\h}$-length, i.e., $|\underline{\h}(\mathcal{M}_1)| \leq s $ for some integer $s>0$. $\mathcal{M}_{1}$ is a disjoint union  $\mathcal{M}_{1} = \mathcal{A} \biguplus \mathcal{B}$, where:

\begin{itemize}

 \item $\mathcal{A}$ consists of all lattices $M$ in $\mathcal{M}_{1}$ such that 
 $|\underline{\h}(M)|\leq b$,
% Symbolically (although not quite correctly), $\mathcal{A} : = \mathcal{M}_{1} \cap \mathcal{L}_{1}$;

 \item $\mathcal{B}$ consists of all lattices $M$ in $\mathcal{M}_{1}$ such that
$|\underline{\h}(M)| > b$.

 %\item $\mathcal{C}$ consists of all lattices $M$ in the complement of  $\mathcal{A}$ in $\mathcal{M}_{1}$ such that
%$|\underline{\h}(M)| > b$.

\end{itemize}

If there are infinitely many pairwise non-isomorphic lattices in
$\mathcal{B}$, then we can find an integer $t$, $b<t\leq s$ with infinitely many objects in $\mathcal{B}$ of the same $\underline{\h}$-length $t$, and our proof is finished. If not, then, by the finiteness principle, the objects of $\mathcal{B}$ appear only in finitely many rows of $\mathbb{M}$.  It follows now that 
$\mathcal{A}$ is of infinite type and in particular is non-empty.  It is also of bounded $\underline{\h}$-length.

We now move to the next column of $\mathbb{M}$ and set 
$\mathcal{M}_2  := \{M_{i,2} | M_{i,1}\in\mathcal{A}\}$. By Corollary~\ref{cor4}, there is an integer $n>0$ such that, $|\underline{\h}(\mathcal{M}_2)| \leq n$ and we apply the same argument as above. We then have that $\mathcal{M}_{2}$ is the union of two disjoint subclasses, denoted by $\mathcal{A}', \mathcal{B}'$. If 
$\mathcal{B}'$ has infinitely many pairwise non-isomorphic lattices, then similar to the above, we can find an integer $t$, $b<t\leq n$, with infinitely many objects in $\mathcal{B}'$ of the same $\underline{\h}$-length $t$. Then our proof is finished. If not, then, exactly as before, $\mathcal{A}'$ is of infinite type and in particular non-empty.

 The foregoing argument can now be applied repeatedly. If the desired family of lattices has not been found at any stage, we can repeat the argument $r$ times, where $r$ is the number from the Harada-Sai lemma \ref{T:Harada-Sai} that depends on $\xx$, which is a contradiction. Thus the desired family of lattices does exist. This finishes the proof of Case 1.
 
\textbf{Case 2}. Assume that the complement of $\mathcal{M}$ in $\mathcal{L}$ is of infinite type and denote it by the same letter
$\mathcal{M}$. Using exactly the same argument as in Case 1, we have a table $\mathbb{M}$. The only difference with the previous case is that each row of 
$\mathbb{M}$ is now finite and starts with an indecomposable projective.  Arguing exactly as in the previous case, we may have three possibilities. First, after finitely many steps, we may have found a requisite infinite family of lattices, which would finish the proof. Secondly, if the process can be repeated infinitely may times, i.e., if the row lengths of~$\mathbb{M}$ are not bounded, we run into a contradiction with the Harada-Sai Lemma, as before, so this option should be discarded. Thirdly, we may run out of nonzero columns of $\mathbb{M}$, which means that~$\mathbb{M}$ has finitely many columns. In that case, since the number of non-isomorphic lattices (i.e., indecomposable projectives) in the first column is finite, Remark~\ref{R:fin-many} applied repeatedly shows that the number of non-isomorphic lattices in the last column is also finite. But this contradicts the assumption. This finishes the proof of Case 2 and of the theorem.
\end{proof}

\begin{remark}
 The foregoing proof shows that the requisite family of lattices can be found in a single column of the table $\mathbb{M}$.
\end{remark}

Applying the above theorem successively, leads us to state and prove a 
mCM type Brauer-Thrall theorem for lattices.

\begin{cor}\label{C:1.5}
Let $\mathcal{L}$ be a subcategory of $\Lambda$-$\textbf{lat}$ of bounded 
$\underline{\h}$-length containing infinitely many non-isomorphic indecomposable objects. Then $\Lambda$ has strongly unbounded lattice type.
\end{cor}

\begin{proof}
By assumption there is an integer $b>0$ such that $|\underline{\h}(\mathcal{L})| \leq b$. Without loss of generality, we may assume that $\mathcal{L}$  consists  of indecomposable objects of $\Lambda$-$\textbf{lat}$ of the same 
$\underline{\h}$-length. A repeated application of Theorem~\ref{T:BT-2} shows
%$b_1$. By~, there is an integer $b_2>b_1$ and a subcategory $\mathcal{L}_2$ consisting of infinitely many of non-isomorphic indecomposable lattices of the same $\underline{\h}$-length $b_2$.
%By using again and again of Theorem \ref{T:BT-2}, yields 
that $\Lambda$ has strongly unbounded lattice type.
\end{proof}
\begin{cor}
If $\Lambda$ is of uncountable $\textbf{lat}$-type, then $\Lambda$ has strongly unbounded lattice type, i.e., the second Brauer-Thrall theorem holds for
$\Lambda$-lattices.
\qed
\end{cor}

\begin{cor}
Suppose that $(R,\m)$ is an isolated singularity of uncountable $\mcm$ type. Then the second Brauer-Thrall theorem holds for $\mcm$ $R$-modules.
\end{cor}
\begin{proof}
Since $R$ is an isolated singularity, the category of $\mcm$ $R$-modules coincides with the category of $R$-lattices. The desired result now follows from the previous corollary.
 \end{proof}
 
\begin{prop}\label{R[[x]]}
Suppose $R$ is an $R$-order. If $R$ has strongly unbounded lattice type, then so does $R[[x]]$.
\end{prop}

\begin{proof}
Let $\F := \{M_i\}_{i\in I}$ be an infinite family of non-isomorphic indecomposable $R$-lattices of the same $\underline{\h}$-length. By Lemma~\ref{L:bound-sop}, there is a faithful s.o.p. $\xx$ for this family. Furthermore, by Lemma~\ref{L:betti}, the $M_{i}$ have bounded betti numbers. Setting $L_i:=M_i\otimes_R R[[t]]$ and using the fact that $R[[t]]$ is a faithfully flat $R$-module, we have that the  $L_i$  are non-isomorphic indecomposable lattices over $R[[t]]$ with bounded betti numbers and $\xx$ is a faithful s.o.p. for the $L_{i}$. Lemma~\ref{L:lem12} now shows that the $L_{i}$ have bounded $\underline{\h}$-length and 
Corollary~\ref{C:1.5} finishes the proof.
\end{proof}

\begin{prop}
Suppose $R$ is an $R$-order and let $\Gamma$ be a finite group.
If the category of lattices over $R$ has strongly unbounded lattice type, then so does the category of lattices over $R\Gamma$. 
\end{prop}

\begin{proof}
Since $R\Gamma$ is a faithfully flat $R$-module, the argument given in the proof of Proposition~\ref{R[[x]]} yields the desired result.
%that
%the category of lattices over $R\Gamma$ has strongly unbounded lattice type.
\end{proof}

We refer to \cite [Chapter II]{assem2006elements}, for the terminology and notation related to path algebras of quivers and their  representations.  Let  $Q=(Q_0, Q_1, s, t)$ be a quiver, where $Q_0$ and $Q_1$ are, respectively, the sets of vertices and arrows of $Q$, and $s, t:Q_1\lra Q_0$ are the two maps which associate to any arrow $\alpha\in Q_1$ its source $s(\alpha)$ and its target $t(\alpha)$.  A vertex $v \in Q_0$ is called a \texttt{sink} if there is no arrow that starts at $v$. A quiver $Q$ is said to be finite if both $Q_0$ and $Q_1$ are finite sets. A path of length $l\geq 1$ with source $a$ and target $b$ (from $a$ to $b$) is a sequence of arrows $\alpha_1, \alpha_2, \cdots, \alpha_l$ where $\alpha_i \in Q_1$, for all $1\leq i \leq l$, such that  $s(\alpha_1)=a$, $s(\alpha_i) = t(\alpha_{i-1})$ for all $1 < i \leq l$, and $t(\alpha_l) = b$. Vertices are viewed as paths of length zero. A path of length $l\geq 1$ is called a \texttt{cycle} if its source and target coincide. The quiver $Q$ is said to be \texttt{acyclic} if it contains no cycles.

The quiver $Q$ can be viewed as a category whose objects are its vertices  and morphisms are all the paths in $Q$. A representation $X$ of $Q$ by finitely generated $R$-modules is a covariant functor $X:Q  \lra \md R$. Such a representation is determined by giving a module $X_v$ for each vertex $v$ of $Q$ and a homomorphism $\varphi_\alpha:X_v\lra X_w$ for each arrow 
$\alpha:v\lra w$ of $Q$. Accordingly, it can be denoted by 
$(X_v, \varphi_\alpha)_{v\in Q_0, \alpha\in Q_1}$ or simply 
$X=(X_v, \varphi_\alpha)$.  A morphism between two representations $X$ and $Y$ is a natural transformation between them. Thus a representation of $Q$ by finitely generated modules over a ring $R$ form a category, denoted by 
$\rep (Q, R)$. If $Q$ is finite and acyclic, then this category is equivalent to the category of finitely generated $RQ$-modules, where $RQ$ is the path algebra of $Q$. 
 
\begin{prop}
Let $Q$ be a finite acyclic quiver and $\Lambda=RQ$. If $R$ is an $R$-order of strongly unbounded lattice type, then so is $\Lambda$. 
%has strongly unbounded lattice type.
\end{prop}

\begin{proof}
Clearly, $\Lambda$ is an $R$-order. By assumption, there is an infinite family 
$\{M_i\}_{i\in I}$ of indecomposable non-isomorphic $R$-lattices such that 
$l_R(\underline{\h}(M_{i})) = b$ for some $b>0$. Since $Q$ is finite and acyclic, there is a sink $u\in Q_0$. For each ${i\in I}$, let $L_i$ be the representation 
$((L_i)_a,\varphi_{\alpha})$ defined as follows:
\[ 
(L_i)_a = \left\{ {\begin{array}{ll}
M_i  & \text{ if $a = u$,}\\
0 & \text{otherwise,}
\end{array}} \right.
\]
and $\varphi_{\alpha}=0$ for all $\alpha\in Q_1$. Obviously, the $L_i$ are 
$\Lambda$-lattices. Since the $M_{i}$ are indecomposable as $R$-modules, the 
$L_{i}$ are indecomposable over $\Lambda$. 
%Since $\underline{\End}_R(M_i)\simeq\underline{\End}_{\Lambda}(L_i)$ and $L_i^,$s are without any projective direct summand, $L_i^,$s are indecomposable $\Lambda$-modules.
Recalling the definition of $\underline{\h}$-length, it is easy to see that
$\underline{\h}(L_i)=\underline{\Hom}_{\Lambda}(L_i, L_i\oplus G')$,
where $G'$ is the $RQ$-lattice $((G')_a, \psi_\alpha)$ defined as follows:
\[ (G')_a= \left\{ {\begin{array}{ll}
G  & \text{ if $a=u$,}\\
0 & \text{ otherwise,}
\end{array}} \right. \]
where  $G$ is a minimal $\textbf{lat}$-approximation of the simple $R$-module $k$ and $\psi_{\alpha}=0$ for all $\alpha\in Q_1$. Since $u$ is a sink, it is not difficult to see that we have isomorphisms 
\[\underline{\Hom}_R(M_i,M_j)\simeq\underline{\Hom}_{\Lambda}(L_i,L_j) \quad  \textrm{and} \quad \underline{\Hom}_R(M_i,G)\simeq\underline{\Hom}_{\Lambda}(L_i,G')
\]
It follows that $l_{\Lambda}(\underline{\h}(L_{i})) =b$.
\end{proof}

\section{Hypersurfaces of bounded and strongly unbounded lattice type}

As we mentioned in the introduction, if $R$ is an $R$-order, then 
$R$-$\textbf{lat}$ = mCM$_0$, where mCM$_0$ is the category of  all maximal Cohen-Macaulay modules that are free on the punctured spectrum of $R$. Now let $S:=k[[x_0, \cdots, x_d]]$ be a ring of formal power series over $k$,  where $k$ is a field of characteristic 0. We fix a nonzero  element $f\in \mathfrak{n}^2$, where $\mathfrak{n} := (x_{0}, \ldots, x_{d})$ is the unique maximal ideal of $S$, 
and set $R:=S/(f)$. The so-called \texttt{double branched cover} $R^{\sharp}$ of $R$ is defined by setting $R^{\sharp} :=S[[u]]/(f+u^2)$, where $u$ is a new letter. Notice that $S$ is a subring of $R^{\sharp}$, making the latter a free $S$-module of rank 2 generated by (the classes of) $1$ and $u$.
%indeterminate over $S$. 
Since $R \simeq R^{\sharp}/(u)$, any $R$-module can be viewed as an $R^{\sharp}$-module. Since $R$ and $ R^{\sharp}$ are Gorenstein rings, they are also orders. There are functors between the categories of mCM $R$-modules and mCM $R^{\sharp}$-modules, see \cite[Chapter 12]{yoshino1990cohen}. For an mCM $R$-module $M$, define $M^{\sharp}$ as the  first syzygy module $\Omega_{R^{\sharp}}^1(M)$ of $M$ viewed as an $R^{\sharp}$-module. On the other hand, if $N$ is an mCM  $R^{\sharp}$-module, we set $\overline{N} := N/uN$. In this section, motivated results of  Kn\"{o}rrer \cite{knorrer1987cohen} and Buchweitz-Greuel-Schreyer \cite[Theorem A]{buchweitz1987cohen}, we show that $R$ and  $R^{\sharp}$ are simultaneously of bounded (or strongly unbounded) $\textbf{lat}$-type. As  a consequence, we show that $R$ has strongly unbounded lattice type whenever $k$ is infinite, see Theorem \ref{21}.

We begin with a series of general observations. To avoid complicated notation, elements of~$R$ and $R^{\sharp}$ will be denoted by their representatives in $S$ and $S[[u]]$. First, notice that $S$, being a regular local ring, is a UFD, and hence $S[[u]]$ is also a UFD. Since $f$ is expressed in terms of the~$x_{i}$ only, we can easily deduce
\begin{lemma}\label{L:nzd}
The images of both $u$ and $f$ in $R^{\sharp}$ are non-zerodivisors. \qed
\end{lemma}

Now suppose that $M$ is a stable mCM $R$-module. Thus $M$ has a minimal projective resolution coming from a reduced matrix factorization of $f$ in $S$. Namely, there are square matrices~$\phi$ and~$\psi$ of size, say, $n$ with entries in the maximal ideal $\mathfrak{n}$ of $S$ such that 
\[
\phi\psi = f 1_{n} = \psi\phi
\]
and the corresponding free resolution of $M$ is periodic of period at most 2 (see~\cite{Eis80}):
\[
\cdots \st{\psi}\lra R^{n} \st{\phi}\lra R^{n}\st{\psi}\lra R^{n} \st{\phi}\lra R^{n} \lra M \lra  0,
\]
where $\phi$ and $\psi$ are actually the classes of $\phi$ and $\psi$ modulo~$(f)$.

Our immediate goal is to construct a minimal free resolution of $M$ over 
$R^{\sharp}$. Let $\mathbb{P}$ denote the two-term complex
$R^{n} \overset{\phi}{\to} R^{n}$. Lifting it to $R^{\sharp}$, we have a two-term
complex $R^{\sharp n} \overset{\phi}{\to} R^{\sharp n}$, which we denote by 
$\mathbb{P}^{\sharp}$. By Lemma~\ref{L:nzd}, we have a short exact sequence of complexes
\[
0 \lra \mathbb{P}^{\sharp} \overset{u}{\lra} \mathbb{P}^{\sharp} \lra \mathbb{P} \lra 0.
\]
This gives a quasi-isomorphism between the mapping cone of $u$ and 
$\mathbb{P}$. Since the homology of~$\mathbb{P}$ is concentrated in degree zero and is isomorphic to $M$, we have a free presentation of $M$ 
over~$R^{\sharp}$:
\begin{equation}\label{E:presentation}
\xymatrix
	{
	\ldots \ar[r]
	& R^{\sharp n} \oplus R^{\sharp n} 
	\ar[r]^>>>>>{\left[ 
		\begin{smallmatrix}
		-u & \phi
		\end{smallmatrix}
	\right]}
	& R^{\sharp n} \ar[r]
	& M \ar[r]
	& 0.
	}
\end{equation}
 
\noindent Now notice that the reduced matrix factorization $(\phi, \psi)$ of $f \in S$ gives rise to a reduced matrix factorization $(\Phi, \Psi)$ of $f +u^{2} \in S[[u]]$, where 
 \[
 \Phi =
 \begin{bmatrix}
 \psi & u \\
 -u & \phi
\end{bmatrix}
 \quad
 \textrm{and}
 \quad
 \Psi =
 \begin{bmatrix}
 \phi & -u \\
 u & \psi
\end{bmatrix}
 \]
and, by yet another abuse of notation, $u$ denotes the scalar matrix with $u$ on the diagonal. This allows to extend~\eqref{E:presentation} to a complex 
\begin{equation}\label{E:resolution}
\xymatrix
	{
	\ldots \ar[r]^>>>>>{\Psi}
	&R^{\sharp n} \oplus R^{\sharp n} \ar[r]^{\Phi}
	& R^{\sharp n} \oplus R^{\sharp n} \ar[r]^{\Psi}
	& R^{\sharp n} \oplus R^{\sharp n} 
	\ar[r]^>>>>>{\left[ 
		\begin{smallmatrix}
		-u & \phi
		\end{smallmatrix}
	\right]}
	& R^{\sharp n} \ar[r]
	& M \ar[r]
	& 0.
	}
\end{equation}

\begin{lemma}
 The above complex is a minimal free resolution of the $R$-stable module $M$ viewed as an $R^{\sharp}$-module.
\end{lemma}

\begin{proof}
Since $(\Phi, \Psi)$ is a matrix factorization, we only need to show the exactness in degree 1, i.e., that the kernel of $[-u~\phi]$ is contained in the image of $\Psi$. The latter is equal to the kernel of $\Phi$. Thus we need to show that for 
$(a,b) \in R^{\sharp n} \oplus R^{\sharp n}$ such that $-ua +\phi(b) = 0$ we also have $\psi(a) + ub = 0$. Since $f$ is a non-zerodivisor on $R^{\sharp}$ and 
$f1_{n} = \phi\psi$, we have that $\psi$ is monic on~$R^{\sharp n}$ and, therefore,
\[
\begin{aligned}
 -ua +\phi(b) = 0 & \Leftrightarrow -u\psi(a) + \psi\phi(b) = 0 \\
 & \Leftrightarrow -u\psi(a) + f(b) = 0 \\
 & \Leftrightarrow -u\psi(a) - u^{2}(b) = 0 \\
 & \Leftrightarrow \psi(a) + u(b) = 0.
\end{aligned}
\]
In the first equivalence we used the fact that $u$ is a scalar matrix, and in the last -- that $u$ is a non-zerodivisor. This shows that~\eqref{E:resolution} is exact. Since all its matrix entries are in the maximal ideal of $R^{\sharp}$, this resolution is minimal.
\end{proof}

\begin{remark}\label{R:reduction}
Since $\Omega^{1}_{R^{\sharp}}M \simeq \cok\, \Psi$,  
$\Omega^{1}_{R}M$ is stable and 
\[
\overline{M^{\sharp}}  \simeq M \oplus 
\Omega^{1}_{R}M
\]
as $R$-modules. Moreover, since $\Phi$ and $\Psi$ are obviously conjugate, 
$\Omega^{1}_{R^{\sharp}}M^{\sharp} \simeq M^{\sharp}$.
\end{remark}

\begin{lemma}\label{lem13}
If $M$ and $C$ are mCM $R$-modules, then there is an $R^{\sharp}$-isomorphism
\[
\underline{\Hom}_{R^{\sharp}}(M^{\sharp},C^{\sharp})\simeq\underline{\Hom}_R(M\oplus\Omega_R^1 M,C).
\]
\end{lemma}

\begin{proof}
%Since $M$ and $C$ are stable mCM modules over the hypersurface ring 
%$S/(f))$, their minimal free resolutions come from matrix factorizations of $f$ and 
%are periodic of period at most 2 (see~\cite{Eis80}). Thus we have exact complexes 
If $M$ or $C$ is free, the lemma is trivially true. Thus assume that $M$ and $C$ are stable. Let
\[
\cdots \st{\psi}\lra R^{n} \st{\phi}\lra R^{n}\st{\psi}\lra R^{n} \st{\phi}\lra R^{n} \lra M \lra  0
\]
and
\[
\cdots \st{\theta}\lra R^{s} \st{\eta}\lra R^{s}\st{\theta}\lra R^{s} \st{\eta}\lra R^{s} \lra C \lra  0
\]
be minimal free resolutions. 
%One may obtain the following commutative diagram which is similar to the diagram appearing in the proof of~\cite[Lemma 12.3]{yoshino1990cohen}
% \[
% \xymatrix@C-0.1pc@R-.8pc
%	{
%	R^{\sharp^{n}} 	\ar[r]^{\psi} \ar[d]^{u}  
%	& R^{\sharp^{n}} \ar[d]^{u} \ar[r]^{\varphi} 
%	& R^{\sharp^{n}}  \ar[d]^{u}& 
%\\	R^{\sharp^{n}} \ar[r]^{\psi} \ar[d] 
%	& R^{\sharp^{n}} \ar[d] \ar[r]^{\varphi} 
%	& R^{\sharp^{n}} \ar[d]
%	& 
%\\
%	R^{n}\ar[r]^{\psi} \ar[d] 
%	& R^{n} \ar[r]^{\varphi} \ar[d]
% 	& R^{n} \ar[r] \ar[d]& M \ar[r] 
%	& 0,&
%\\
%	0
%	& 0
%	& 0
% 	}
% \]
%where the columns and the bottom row are exact, while the first and the second rows are not even complexes.
As above, we have minimal free resolutions over $R^{\sharp}$:
%Now, by chasing the diagram, we get a free presentation of $M$ as an $R^{\sharp}$-module
\begin{equation}\label{E:M-res}
 \xymatrix
	{
	\ldots \ar[r]^>>>>>{\Psi}
	&R^{\sharp n} \oplus R^{\sharp n} \ar[r]^{\Phi}
	& R^{\sharp n} \oplus R^{\sharp n} \ar[r]^{\Psi}
	& R^{\sharp n} \oplus R^{\sharp n} 
	\ar[r]^>>>>>{\left[ 
		\begin{smallmatrix}
		-u & \phi
		\end{smallmatrix}
	\right]}
	& R^{\sharp n} \ar[r]
	& M \ar[r]
	& 0
	}
\end{equation}
%\[
%R^{{\sharp}^{(n)}}\oplus R^{{\sharp}^{(n)}}\st{\overline{\psi}}\lra R^{{\sharp}^{(n)}}\oplus R^{{\sharp}^{(n)}}\st{[\varphi~~u]}\lra R^{{\sharp}^{(n)}}\lra M\lra 0.
%\]
and 
\[
\xymatrix
	{
	\ldots \ar[r]^>>>>>{\Theta}
	& R^{\sharp s} \oplus R^{\sharp s} \ar[r]^{H}
	& R^{\sharp s} \oplus R^{\sharp s} \ar[r]^{\Theta}
	& R^{\sharp s} \oplus R^{\sharp s} 
	\ar[r]^>>>>>{\left[ 
		\begin{smallmatrix}
		-u & \eta
		\end{smallmatrix}
	\right]}
	& R^{\sharp s} \ar[r]
	& C \ar[r]
	& 0.
	}
\]
%where 
% \[
% H = 
% \begin{bmatrix}
% \theta & u \\
% -u & \eta
%\end{bmatrix}
% \quad
% \textrm{and}
% \quad
% \Theta =
% \begin{bmatrix}
% \eta & -u \\
% u & \theta
%\end{bmatrix}
% \]
%
%We apply the same argument to get a free presentation of $C$ as an $R^{\sharp}$-module
%\[R^{{\sharp}^{(s)}}\oplus R^{{\sharp}^{(s)}}\st{\overline{\theta}}\lra R^{{\sharp}^{(s)}}\oplus R^{{\sharp}^{(s)}}\st{[\eta~~u]}\lra R^{{\sharp}^{(s)}}\lra C\lra 0
%\]
%where
%$\overline{\psi}=\tiny {\left[\begin{array}{ll} \psi & {-uI} \\
%uI& \varphi \end{array} \right]}$ and $\overline{\theta}= \tiny {\left[\begin{array}{ll} \theta &-uI \\
%uI&\eta \end{array} \right]}$.
%It should be noted that $\cok\overline{\psi}=M^{\sharp}$ and $\cok\overline{\theta}=C^{\sharp}$.

An object of interest to us is $\Hom_{R^{\sharp}}((\Omega^{1}_{R}M)^{\sharp}, C)$. Since $u$ acts on $C$ by zero, any homomorphism 
$f \in \Hom_{R^{\sharp}}((\Omega^{1}_{R}M)^{\sharp}, C)$ factors through the reduction of $(\Omega^{1}_{R}M)^{\sharp}$ modulo $u$. By 
Remark~\ref{R:reduction}, with $M$ replaced by $\Omega^1_{R} M$ and $\Psi$ replaced by $\Phi$, the result of that reduction is isomorphic to $M\oplus\Omega_R^1 M$, and we have an induced homomorphism 
$g \in \Hom_R(M\oplus\Omega_R^1 M, C)$. Clearly, if $f$ factors through a projective (i.e., free) module, then so does $g$. Conversely, given 
$g \in \Hom_R(M\oplus\Omega_R^1 M, C)$, we can compose it with the reduction homomorphism to get some $f \in \Hom_{R^{\sharp}}((\Omega^{1}_{R}M)^{\sharp}, C)$. If $g$ factors through some $R^{t}$, then the same is true for 
$g\pi$, where $\pi : (\Omega^{1}_{R}M)^{\sharp} \to M \oplus\Omega_R^1 M$ is the reduction homomorphism. The short exact sequence 
$0 \lra R^{\sharp t} \overset{u}{\lra} R^{\sharp t} \lra R^{t} \lra 0$ is an mCM approximation. Since $(\Omega^{1}_{R}M)^{\sharp}$ is mCM, any map from it to $R^{t}$ lifts to $R^{\sharp t}$, showing that $f$ factors through a free 
$R^{\sharp}$-module. As a consequence, we have an isomorphism 
\begin{equation}\label{E:first-half}
 \underline{\Hom}_{R^{\sharp}}((\Omega^{1}_{R}M)^{\sharp}, C) \simeq
\underline{\Hom}_R(M\oplus \Omega^{1}_{R} M, C).
\end{equation}

%We have
%\[
%\begin{aligned}
%\Hom_{R^{\sharp}}((\Omega^{1}_{R}M)^{\sharp}, C) & \simeq 
%\Hom_{R^{\sharp}}((\Omega^1_{R} M)^{\sharp}, \Hom_R(R^{\sharp}/(u),C)) \\
%& \simeq \Hom_R((\Omega^1_{R} M)^{\sharp} \otimes_{R^{\sharp}}
%R^{\sharp}/(u),C) \\
%& \simeq \Hom_R(M \oplus\Omega_R^1 M, C)
%\end{aligned}
%\]
%where the last isomorphism follows from Remark~\ref{R:reduction} with $M$ replaced by $\Omega^1_{R} M$ and $\Psi$ replaced by $\Phi$.
%\cite[Proposition 12.4]{yoshino1990cohen}.

On the other hand, given $f\in \Hom_{R^{\sharp}}((\Omega_R^1 M)^{\sharp},C)$, we can lift it to the first syzygy modules to obtain a commutative diagram 
%there is an $R$-homomorphism $h$ in
% $\Hom_{R^{\sharp}}(M^{\sharp},C^{\sharp})$  with the
%following commutative diagram;
\[
\xymatrix
	{
	0 \ar[r] 
	& M^{\sharp} \ar[r] \ar[d]_{h}
	& R^{{\sharp n}} \oplus R^{{\sharp n}} \ar[r] \ar[d]
	& (\Omega_R^1 M)^{\sharp} \ar[r] \ar[d]_{f}
	& 0
\\
	0 \ar[r]
	& C^{\sharp} \ar[r]
	& R^{{\sharp s}} \ar[r]
	& C \ar[r]
	& 0
	}
\] 
%{\footnotesize{$$\begin{CD}
%0 @>>>  @>>> \  @>>> \  @>>> 0\\
%& & @V  h VV @V VV @Vf VV & &\\ 0 @>>>  @> >> @>>> C
%@>>> 0.\end{CD}$$}}
Here $h$ is determined by $f$ uniquely up to a map factoring through a projective. 
%  If  $f$ factors through a free $R$-module $F$, then using the fact that $0 \lra F^{\sharp} \st{u}\lra F^{\sharp} \st{nat}\lra F \lra 0$ is a projective resolution of $F$ as $R^{\sharp}$-module,  one can deduce that   $h$ factors through a projective $R^{\sharp}$-module.
Conversely, given $h\in\Hom_{R^{\sharp}}(M^{\sharp},C^{\sharp})$, since $R^{\sharp}$ is a Gorenstein ring and $(\Omega_R^1 M)^{\sharp}$ is an mCM $R^{\sharp}$-module, we can extend $h$ to obtain $f\in\Hom_{R^{\sharp}}((\Omega_R^1 M)^{\sharp},C)$,   which is unique up to a map factoring through a projective. Thus we  have an isomorphism
\[
\underline{\Hom}_{R^{\sharp}}(M^{\sharp},C^{\sharp}) \simeq 
\underline{\Hom}_{R^{\sharp}}((\Omega^{1}_{R}M)^{\sharp}, C).
\]
Comparing this with~\eqref{E:first-half} yields the desired result.
\end{proof}

\begin{lemma}\label{lem14}
In the above notation, the functors $M \mapsto M^{\sharp}$ and 
$N \mapsto \overline{N}$ induce functors between the categories of lattices over $R$
and over $R^{\sharp}$.
%With the notations above, the following assertions hold;\\
%$1)$ If $M$ is a lattice over $R$, then $M^{\sharp}$ is a lattice over $R^{\sharp}$.\\
%$2)$ If $N$ is a lattice over $R^{\sharp}$, then $\overline{N}$ is a lattice over $R$.
\end{lemma}
\begin{proof}
If $M$ is an $R$-lattice, then so is $\Omega_R M$, and 
$\underline{\Hom}_R(\Omega_R^1 M\oplus M, M)$ is an artinian $R$-module. Hence it is artinian as an $R^{\sharp}$-module. As~\eqref{E:M-res} shows, 
$M^{\sharp}$ is an mCM $R^{\sharp}$-module. 
%see \cite[page 108]{yoshino1990cohen} for details.
Hence, by Lemma \ref{lem13},  $\underline{\Hom}_{R^{\sharp}}(M^{\sharp}, M^{\sharp})$ is an artinian $R^{\sharp}$-module, which shows that $M^{\sharp}$ is projective on the punctured spectrum of $R^{\sharp}$, and therefore $M^{\sharp}$ is an $R^{\sharp}$-lattice.

Assume now that $N$ is an $R^{\sharp}$-lattice. If $N$ is free, the claim is obvious. Thus assume that~$N$ is stable. In this case, $N$ has a minimal free resolution 
\begin{equation}\label{resolution-P}
 \mathbb{P}   \quad \quad \ldots \overset{\Phi}{\lra} P \overset{\Psi}{\lra} P \overset{\Phi}{\lra} P \lra N \lra 0
\end{equation}
coming from a reduced matrix factorization of $f +u^{2}$ over $S[[u]]$. By Lemma~\ref{L:nzd}, $\overline{N}$ is an mCM $R$-module. Its minimal projective resolution as an $R^{\sharp}$-module is given by the mapping cone 
\[
\xymatrix
	{
	\ldots \ar[r]
	& P^{2} \ar[r]^{\left[ \begin{smallmatrix} -\Psi & 0\\ u & \Phi \end{smallmatrix} \right]}
	& P^{2} \ar[r]^{\left[ \begin{smallmatrix} -\Phi & 0\\ u & \Psi \end{smallmatrix} \right]}
	& P^{2} \ar[r]^{\left[ \begin{smallmatrix} u & \Phi \end{smallmatrix} \right]}
	& P \ar[r]
	& \overline{N} \ar[r]
	& 0
	}
\] 
of the injective chain map $\mathbb{P} \overset{u}{\lra} \mathbb{P}$. 
Accordingly, we have a minimal free resolution  
\[
\xymatrix
	{
	 \mathbb{Q} \quad
	& \ldots \ar[r]
	& P^{2} \ar[r]^{\left[ \begin{smallmatrix} -\Psi & 0\\ u & \Phi \end{smallmatrix} \right]}
	& P^{2} \ar[r]^{\left[ \begin{smallmatrix} -\Phi & 0\\ u & \Psi \end{smallmatrix} \right]}
	& P^{2} \ar[r]
	& (\overline{N})^{\sharp} \ar[r]
	& 0
	}
\] 
of $(\overline{N})^{\sharp} = \Omega^{1}_{R^{\sharp}}\overline{N}$. This resolution is the mapping cone of the chain map 
$\mathbb{P} \to \mathbb{P}_{\geq 1}$ given by multiplication by~$u$. Here the subscript $\geq 1$ indicates truncation in degree 1. Thus we have a short exact sequence 
\[
0 \lra  \mathbb{P}_{\geq 1} \lra  \mathrm{Con}(u) \lra \mathbb{P}[-1] \lra 0
\]
of complexes. The corresponding long homology exact sequence degenerates into a short exact sequence 
\begin{equation}\label{split-sequence}
 0 \lra \Omega^{1}N \lra (\overline{N})^{\sharp} \lra N \lra 0
\end{equation}
which shows that $(\overline{N})^{\sharp}$ is an $R^{\sharp}$-lattice. Therefore, 
$\underline{\Hom}_{R^{\sharp}}((\overline{N})^{\sharp}, (\overline{N})^{\sharp})$
 is an artinian $R^{\sharp}$-module. Applying Lemma \ref{lem13}, we conclude that $\underline{\Hom}_R(\overline{N} \oplus \Omega_R^1\overline{N}, \overline{N})$ is an artinian $R^{\sharp}$-module and so it is artinian as an $R$-module. Hence, $\overline{N}$ is an $R$-lattice, as required.
\end{proof}

\begin{remark}\label{R:0-homotopy}
 The argument in~\cite[Lemma 8.17]{leuschke2012cohen} shows that the resolution $\mathbb{P}$ from~\eqref{resolution-P} can be chosen in the form 
 $\Phi = u \otimes_{S} 1_{N} +1_{S[[u]]} \otimes_{S} \phi$ and $\Psi = u \otimes_{S} 1_{N} -1_{S[[u]]} \otimes_{S} \phi$, where $\phi$ is the endomorphisms of $N$ viewed as an $S$-module. It is easy now to show that the chain map $\mathbb{P} \to \mathbb{P}_{\geq 1}$ is 0-homotopic (just take the map ${\frac{1}{2}\left[ \begin{smallmatrix} -1 & 0\\ 1 & 1 \end{smallmatrix} \right]}$ as the homotopy). As a consequence, the sequence~\eqref{split-sequence} splits.
\end{remark}

Using the argument given in the proof of \cite[Theorem 12.5]{yoshino1990cohen}
and applying Remark~\ref{R:0-homotopy}, we have the next result. 

\begin{theorem}\label{cor7}
 $R$ is of finite $\textbf{lat}$-type if and only if  $R^{\sharp}$ is of finite $\textbf{lat}$-type. \qed
\end{theorem}

As an immediate consequence of the above theorem, we have
\begin{cor}\label{cor6}
$R$ is of bounded lattice type if and only if $R^{\sharp}$ is of bounded lattice type.
\end{cor}

\begin{proof}
Apply Theorems \ref{bounded} and \ref{cor7}.
\end{proof}

\begin{lemma}\label{L:mod-u}
Let $N$ be an mCM $R^{\sharp}$-module. The canonical surjection $\pi : N \to \overline{N}$ induces an isomorphism 
\[
(\underline{\pi, k}) : \underline{\Hom}_R(\overline{N}, k) \lra
\underline{\Hom}_{R^{\sharp}} (N,k)
\]
%
%
%  In view of the above isomorphisms, to prove the lemma it suffices to show that the induced map 
%
%is an isomorphism. 
%
%
%
%
\end{lemma}

\begin{proof}
 To show that $(\underline{\pi, k})$ is well-defined, pick an arbitrary map
 $f : \overline{N} \to k$ that factors through a projective, i.e., some $R^{n}$. Then the same is true for $f\pi$. Since $R$ is an $R^{\sharp}$-module of projective dimension one, its mCM approximation is $R^{\sharp n}$. Since $N$ is mCM, we now have that $f\pi$ factors through $R^{\sharp n}$. We have just shown that
 $(\underline{\pi, k})$ is well-defined. Since $u$ annihilates $k$, any map $N \to k$ extends to a map $\overline{N} \to k$. It follows that $(\underline{\pi, k})$ is epic. Finally, if the map $f\pi$ factors through some $R^{\sharp}$, then reducing it modulo $u$, we recover $f$ and, at the same time, have a factorization of $f$ through $R^{n}$, which shows that $(\underline{\pi, k})$ is monic.
\end{proof}

\begin{lemma}\label{lem20}
Let  $g : G \lra k$  and $h : H \lra k$ be minimal $\textbf{lat}$-approximations of $k$ as an $R$-module and as an $R^{\sharp}$-module, respectively. Then, for any  $N\in R^{\sharp} $-$\textbf{lat}$,
\[
\underline{\Hom}_{R^{\sharp}}(N,H)\simeq \underline{\Hom}_R(\overline{N}, G).
\]
In particular, for any $M\in R$-$\textbf{lat}$,
\[
\underline{\Hom}_{R^{\sharp}}(M^{\sharp},H)\simeq \underline{\Hom}_R(M\oplus\Omega_R^1 M, G).
\]
\end{lemma}
\begin{proof} Since $R$ and $R^{\sharp}$ are Gorenstein rings, Corollary~\ref{C:fpd} shows that $\Ker\, g$ and $\Ker\, h$ are of finite projective dimension.
% as $R$-module and $R^{\sharp}$-module, respectively. 
Given any mCM $R$-module $M$, we claim that the induced homomorphism  
$\rho : \underline{\Hom}_{R}(M,G) \lra \underline{\Hom}_R(M,k)$ is an isomorphism. Indeed, by Corollary~\ref{C:fpd}, $g$ is also an mCM approximation, and therefore any map from $M$ to $k$ lifts over $g$. It follows that $\rho$ is epic. To show that $\rho$ is monic, we first observe that the functor covariant Hom modulo projectives is half-exact. If $f \in {\Hom}_{R}(M,G)$ is such that $\rho([f]) = 0$ (i.e., $g f$ factors through a projective), then, by the half-exactness, $f$ factors through the kernel of the approximation up to a map factoring through a projective. By Corollary~\ref{C:fpd}, that kernel is of finite projective dimension, and therefore its mCM approximation is projective. It follows that the map $M \to G$ factoring through $\Ker\, g$ also factors through a projective. Hence, so does $f$, and $\rho$ is monic. Exactly the same argument shows that, for any mCM $R^{\sharp}$-module $N$, the induced homomorphism 
$\underline{\Hom}_{R^{\sharp}}(N,H) \to \underline{\Hom}_{R^{\sharp}}(N,k)$
is an isomorphism. Combining these isomorphisms with Lemma~\ref{L:mod-u},
we now have
$\underline{\Hom}_{R^{\sharp}}(N,H) \simeq \underline{\Hom}_{R^{\sharp}} (N,k) \simeq \underline{\Hom}_R(\overline{N}, k) \simeq \underline{\Hom}_R(\overline{N}, G)$.
%
%
%\[
%\begin{aligned}
% \underline{\Hom}_{R^{\sharp}}(N,H) & \simeq \underline{\Hom}_{R^{\sharp}} (N,k)\\
% & \simeq \underline{\Hom}_R(\overline{N}, k) \\
% & \simeq \underline{\Hom}_R(\overline{N}, G)
%\end{aligned}
%\]
\end{proof}

\begin{theorem}\label{strongly}
$R$ has strongly unbounded lattice type if and only if $R^{\sharp}$ does.
%has strongly unbounded
%$\textbf{lat}$-type.
\end{theorem}
\begin{proof}
Suppose that $\{M_i\}_{i\in I}$ is an infinite set of non-isomorphic indecomposable $R$-lattices of the same $\underline{\h}$-length. By Lemmas \ref{lem14}, \ref{lem13}, and \ref{lem20}, the family $\F=\{M_i^{\sharp}\}_{i\in I}$ is an infinite  set of $R^{\sharp}$-lattices with bounded $\underline{\h}$-length.
Without loss of generality, we say assume that $\F$ consists of lattices of the same $\underline{\h}$-length. By Remark~\ref{R:reduction}$, \overline{M_i^{\sharp}}\simeq M_i \oplus \Omega_R M_i$, and therefore each $M^\sharp_i$ has at most two indecomposable summands. It follows that the 
number of such pairwise nonisomorphic summands is infinite.

%We claim that, there is an infinite set $\{N_i\}_{i\in I}$ of non-isomorphic indecomposable $R^{\sharp}$-lattices such that each $N_i$ is an indecomposable direct summand of $M_i^{\sharp}$. 
%for some $j$. 
%Assuming the opposite and using the assumption that the ${M_i^{\sharp}}$ have the same 
%$\underline{\h}$-length together with  the isomorphisms  $\overline{M_i^{\sharp}}\simeq M_i \oplus \Omega_R M_i$, one can deduce that for infinitely many indices  $i\neq j $,   $M_i\simeq M_j$, which is a contradiction.

The same argument, together with isomorphisms $\overline{N}^{\sharp}\simeq N\oplus\Omega_{R^{\sharp}} N$ for $R^{\sharp}$-lattices $N$, proves the reverse implication.
%By the same process, we can prove that if  $R^{\sharp}$ has strongly unbounded lattice type then $R$ has.
\end{proof}

%\begin{cor}
%Let $R$ be as Theorem \ref{kavazaki}. $R$ is of finite mCM type if and only if there exists $b\in\mathbb{N}$ such that
%$l_R(\Ext_R^d(R/\m^n, R/\m^n\oplus R/\m))<b$, for each $n$.
%\end{cor}
\begin{theorem} \label{22}
Let $(R,\m, k)$ be a one-dimensional Cohen-Macaulay local ring with multiplicity 
$e(R)\geq 3$. If $R$ contains the residue field $k$ and $k$ is  infinite, then $R$ has strongly unbounded lattice type.
\end{theorem}

\begin{proof}
Set $S := \cup_{n\in \mathbb{N}}\End _R(\m^n)$, which is a finite birational extension of $R$ (i.e., an intermediate ring between $R$ and its total quotient ring $K$ such that $S$ is finitely generated as an $R$-module), 
see~\cite[Proposition 4.3]{leuschke2012cohen}. Thus $S_{\p}\simeq R_{\p}$ for any non-maximal prime $\p$ of $R$. By~[\textbf{idem}], 
%{leuschke2005hypersurfaces}, 
$\beta_R (S)=e(R)>2$. Let $c$ be the conductor, that is, the largest ideal of $S$ that is contained in $R$. Set $B := S/c$ and $D : = B/\m B$. Suppose that 
$\alpha$ and $\gamma$ are elements of $D$ such that $1, \alpha,$ and
$\gamma$ are linearly independent over $k$. Assume that, for any $t\in k$,  $V_t$ is the $k$-subspace of $D$ generated by $\{ 1, \alpha+t\gamma+ t\lambda\},$ where $\lambda\in k$. 
Consider the following pullback diagram of $R$-modules
\[
\xymatrix
	{
	M_{t} \ar[d] \ar[r] 
	& S \ar[d] 
\\
	V_{t} \ar[r]
	& D
	}
\] 
where $ V_t \lra D$ is the inclusion map and $ S \lra D$ is the canonical map.
As was shown in the proof of~\cite[Theorem 2.5]{leuschke2013brauer}, $M_t$ is an indecomposable mCM $R$-module for any $t$ and all of these modules are pairwise non-isomorphic. The above pullback diagram gives rise to the following  commutative diagram 
\[
\xymatrix
	{
	0 \ar[r]
	& M_{t} \ar[d] \ar[r] 
	& S \ar[d] \ar[r]
	& L_{t} \ar@{=}[d] \ar[r]
	& 0
\\
	0 \ar[r]
	& V_{t} \ar[r]
	& D \ar[r]
	& L_{t} \ar[r]
	& 0
	}
\] 
with exact rows. Notice that $L_t$ is an artinian $R$-module and hence 
$(M_t)_{\p}\simeq S_{\p}\simeq R_{\p}$, for any non-maximal prime ideal $\p$ of $R$. Therefore, $\F=\{M_t\}_{t\in k}$ is an infinite set of indecomposable lattices. On the other hand, since $\m L_t=0$, there is $s>0$ such that 
$\m^s\Ext_R^1(M_t,-)=0$ for any $t\in k$. Indeed, $S$  is an $R$-lattice  and so there is $s$ such that  $\m^{s}\Ext_R^1(S,-)=0$. Also, $\m\Ext_R^2(L_t,-) = 0$. The claim now follows from the long exact sequence for the top row. Now choose  $x\in\m^s$ such that $x$ is a faithful system of parameters for $\F$. Since, for each $t$,  $e(M_t)\leq e(S)$, $\F$ is of bounded mCM type. Thus, by Lemmas \ref{L:lem12} and \ref{lem11}, $\F$ is of bounded $\underline{\h}$-length, implying that $\F$ has infinitely many non-isomorphic objects of the same $\underline{\h}$-length. By Theorem~\ref{T:BT-2}, $R$ has strongly unbounded lattice type.
\end{proof}

\begin{theorem}\label{21}
Let $S := k[[x_0, \cdots, x_d]]$, where $k$ is an infinite field, $f \in (x_0, \cdots, x_d)^2$, and $R := S/(f)$. If $e(R)\geq 3$, then $R$ has strongly unbounded lattice type.
 \end{theorem}
 
 \begin{proof}
 We induct on $d$. When $d=1$, the result follows from Theorem \ref{22}. Now suppose that $d>1$ and the result has been proved for all values smaller
 than $d$. In view of \cite[Theorem 9.7]{leuschke2012cohen}, we may assume that $f=b+x_d^2$ with $b\in (x_0, \cdots, x_{d-1})^2k[[x_0, x_1, \cdots, x_{d-1}]]$ and  $R=A^{\sharp}$, where $A=k[[x_0, \cdots, x_{d-1}]]/(b)$. As $\dim A= d-1$, the induction assumption implies that~$A$ has strongly unbounded lattice type. The desired result now follows from Theorem~\ref{strongly}.
\end{proof}

\begin{remark}\label{30} Analogs of the Brauer-Thrall conjectures have also been considered for mCM modules over general CM local rings, see for example \cite{dieterich1981representation, yoshino1987brauer, leuschke2013brauer}. In that context, the multiplicity of an mCM module replaces the length of a finitely generated module over an artin algebra.
%Indeed, in the  interpretation of the Brauer-Thrall conjectures in this context,
%the  multiplicity of maximal Cohen-Macaulay modules considered as an invariant instead of the length of finitely generated modules  over artin algebras.
We should  emphasize that the analog of the first Brauer-Thrall theorem for mCM modules fails in general by a counterexample of 
Dieterich~\cite{dieterich1981representation}.
%dealing with the invariant multiplicity. 
In addition, for rings with $e(R)\leq 2$, Bass \cite{bass1963ubiquity} showed that every indecomposable mCM module is isomorphic to an ideal of~$R$, implying that~$R$ has infinitely many non-isomorphic indecomposable mCM modules of the same multiplicity whenever $R$ is not of finite mCM type. But it is not of strongly unbounded mCM type, i.e., there does not exist an infinite sequence $n_1<n_1<\cdots$  of positive integers such that there are, for any~$i$, infinitely many non-isomorphic indecomposable mCM modules of multiplicity $n_i$. This means that the multiplicity-based mCM type Brauer-Thrall and the second Brauer-Thrall theorems do not hold for $R$. However, over several classes of Cohen-Macaulay rings, analogs of these conjectures are  known to be true. Namely, Dieterich \cite{dieterich1981representation} and Yoshino \cite{yoshino1987brauer} showed the validity of the first Brauer-Thrall theorem for complete equicharacteristic CM isolated singularities over a perfect field. Moreover, Leushke and Wiegand \cite{leuschke2013brauer} have proved this theorem when $R$ is an equicharacteristic excellent ring with algebraically closed residue field. In an effort to prove a counterpart of a  result of  Kn\"{o}rrer and Buchweitz-Greuel-Schreyer \cite{knorrer1987cohen, buchweitz1987cohen} on finite mCM type for bounded mCM type, Leushke and Wiegand \cite [Proposition 1.5]{leuschke2005hypersurfaces}, showed that~$R$ and~$R^{\sharp}$ have bounded mCM type simultaneously, where $R=S/(f)$, $S=k[[x_0, \cdots, x_d]]$, and $f$ is a nonzero and non-unit element of $S$.
%We prove the first Brauer-Thrall ( Corollary \ref{cor7}) and one and half Brauer-Thrall (Theorem \ref{strongly}) theorems
%hold for $R$ and $R^{\sharp}$, simultaneously. In addition, we show that $R$ is of bounded lattice type
%if and only if $R^{\sharp}$ is of bounded lattice type; see Theorem \ref{cor6}
\end{remark}

{\bf Acknowledgments}
This project was initiated in October  2015, while the third author was visiting at the Northeastern University. The third author thanks the NU, and in particular Professor Alexander Martsinkovsky, for providing a stimulating research environment. The work of the third author has been in part supported by a grant from  IPM (No. 96130215). He also would like to thank the Center of Excellence for Mathematics (University of Isfahan).

\end{document}